\pdfoutput=1
\documentclass[12pt,a4paper,twoside]{amsart}

\usepackage{amsrefs}
\usepackage{amsbsy}
\usepackage{amssymb}
\usepackage{amscd}
\usepackage{rotating}
\usepackage{amsmath}
\usepackage{bbding}
\usepackage{latexsym}
\usepackage{a4}
\usepackage[english]{babel} 
\usepackage[latin1]{inputenc} 
\usepackage[a4paper]{geometry}
\usepackage[T1]{fontenc}

\usepackage{array}
\usepackage{multirow}

\usepackage{listings}
\usepackage{stmaryrd}
\usepackage{cancel}
\usepackage{graphics} 
\usepackage{graphicx}
\usepackage{tikz}
\usetikzlibrary{arrows,decorations.pathmorphing,decorations.footprints,backgrounds,positioning,fit,calc,through,shapes}
\usepackage{pgf}
\definecolor{MyLightMagenta}{cmyk}{0.1,0.8,0,0.1}
\usepackage[unicode=false,colorlinks=false]{hyperref}

\newcommand{\secA}[1]{\section{#1}}

\def\N{{\mathbb N}}
\def\Z{{\mathbb Z}}

\def\S{{\mathfrak S}}

\def\coloneqq{\mathrel{\mathop:}=}
\def\deff{\mathrel{\mathop:}\Leftrightarrow}

\providecommand{\out}[1]{\mathrm{J}(#1)}
\def\bphi{\boldsymbol{\varphi}}

\def\gcd{\mathrm{gcd\;}}

\def\lcm{\mathrm{lcm\;}}

\def\T{\mathrm{T}}

\def\SYT{\mathrm{SYT}}

\def\0{{\bf 0}}
\def\1{{\bf 1}}

\newcommand{\reft}[1]{Theorem~\ref{#1}}
\newcommand{\reff}[1]{Figure~\ref{fig:#1}}

\newcommand{\refc}[1]{Conjecture~\ref{#1}}

\newcommand{\refp}[1]{Propostion~\ref{#1}}
\newcommand{\refs}[1]{Section~\ref{#1}}

\newcommand{\defstil}[1]{\textup{\textbf{\/#1}}}

\providecommand{\abs}[1]{\left\lvert#1\right\rvert}

\numberwithin{equation}{section}

\theoremstyle{remark}

\newtheorem{defn}{Definition}[section]

\theoremstyle{plain}

\newtheorem{prop}[defn]{Proposition}
\newtheorem{cor}[defn]{Corollary}
\newtheorem{thm}[defn]{Theorem}

\newtheorem{conj}[defn]{Conjecture}

\theoremstyle{remark}
\newtheorem{rem}[defn]{Remark}

\newtheorem{ex}[defn]{Example}
\newtheorem{ex*}{Example}

\linethickness{0.4pt}

\makeindex
\def\FIGpartition{
\begin{tikzpicture}[plineb/.style={draw=gray,thin},plineg/.style={draw=lightgray}]
\node (a) at (0,0) {
\begin{tikzpicture}[scale=0.5]

\draw [plineb] (0.5,5.5) -- (5.5,5.5);
\draw [plineb] (0.5,4.5) -- (4.5,4.5);
\draw [plineb] (0.5,3.5) -- (2.5,3.5);
\draw [plineb] (0.5,2.5) -- (1.5,2.5);
\draw [plineb] (0.5,1.5) -- (1.5,1.5);

\draw [plineb] (1.5,6.5) -- (1.5,0.5);
\draw [plineb] (2.5,6.5) -- (2.5,3.5);
\draw [plineb] (3.5,6.5) -- (3.5,4.5);
\draw [plineb] (4.5,6.5) -- (4.5,4.5);

\draw [very thick] (0.5,6.5)--(5.5,6.5)--(5.5,5.5)--(4.5,5.5)--(4.5,4.5)--(2.5,4.5)--(2.5,3.5)--(1.5,3.5)--(1.5,.5)--(.5,.5)--cycle;

\node at (1,5) {$\bullet$};
\node at (2,5) {$\bullet$};
\node at (3,5) {$x$};
\node at (4,5) {$\bullet$};
\node at (3,6) {$\bullet$};

\node at (-1,3.5) {$\lambda$};
\end{tikzpicture}
};

\node (b) at (6,0) {
\begin{tikzpicture}[scale=0.5]

\draw [plineb] (.5,4.5) -- (6.5,4.5);
\draw [plineb] (.5,3.5) -- (3.5,3.5);
\draw [plineb] (.5,2.5) -- (2.5,2.5);
\draw [plineb] (.5,1.5) -- (2.5,1.5);

\draw [plineb] (1.5,5.5) -- (1.5,0.5);
\draw [plineb] (2.5,5.5) -- (2.5,1.5);
\draw [plineb] (3.5,5.5) -- (3.5,3.5);
\draw [plineb] (4.5,5.5) -- (4.5,4.5);
\draw [plineb] (5.5,5.5) -- (5.5,4.5);

\draw [very thick] (.5,5.5)--(6.5,5.5)--(6.5,4.5)--(3.5,4.5)--(3.5,3.5)--(2.5,3.5)--(2.5,1.5)--(1.5,1.5)--(1.5,.5)--(.5,.5)--cycle;

\node at (-1,3) {$\lambda'$};
\end{tikzpicture}
};

\end{tikzpicture}
}

\def\FIGnps{
\begin{tikzpicture}[post/.style={->,shorten >=1pt,>=stealth',semithick},plineb/.style={draw=gray,thin},plineg/.style={draw=lightgray}]
\node (a) at (0,2.5) {
\begin{tikzpicture}[scale=0.6]

\draw [plineb] (0.5,3.5) -- (3.5,3.5);
\draw [plineb] (0.5,2.5) -- (3.5,2.5);
\draw [plineb] (0.5,1.5) -- (3.5,1.5);
\draw [plineb] (0.5,0.5) -- (1.5,0.5);

\draw [plineb] (0.5,3.5) -- (0.5,0.5);
\draw [plineb] (1.5,3.5) -- (1.5,0.5);
\draw [plineb] (2.5,3.5) -- (2.5,1.5);
\draw [plineb] (3.5,3.5) -- (3.5,1.5);

\draw [very thick] (0.5,3.5)--(3.5,3.5)--(3.5,1.5)--(1.5,1.5)--(1.5,0.5)--(0.5,0.5)--cycle;

\node at (1,3) {$\scriptstyle 4$};
\node at (1,2) {$\scriptstyle 7$};
\node at (1,1) {$\scriptstyle 1$};

\node at (2,3) {$\scriptstyle 6$};
\node at (2,2) {$\scriptstyle 3$};

\node at (3,3) {$\scriptstyle 5$};
\node at (3,2) {$\scriptstyle 2$};

\end{tikzpicture}
};

\node (b) at (3,2.5) {
\begin{tikzpicture}[scale=0.6]

\draw [plineb] (0.5,3.5) -- (3.5,3.5);
\draw [plineb] (0.5,2.5) -- (3.5,2.5);
\draw [plineb] (0.5,1.5) -- (3.5,1.5);
\draw [plineb] (0.5,0.5) -- (1.5,0.5);

\draw [plineb] (0.5,3.5) -- (0.5,0.5);
\draw [plineb] (1.5,3.5) -- (1.5,0.5);
\draw [plineb] (2.5,3.5) -- (2.5,1.5);
\draw [plineb] (3.5,3.5) -- (3.5,1.5);

\node at (1,3) {$\scriptstyle 4$};
\node at (1,2) {$\scriptstyle 7$};
\node at (1,1) {$\scriptstyle 1$};

\node at (2,3) {$\scriptstyle 6$};
\node at (2,2) {$\scriptstyle 3$};

\node at (3,3) {$\scriptstyle 5$};
\node at (3,2) {$\scriptstyle 2$};

\draw [very thick] (2.5,3.5) rectangle (3.5,2.5);

\end{tikzpicture}
};

\node (c) at (6,2.5) {
\begin{tikzpicture}[scale=0.6]

\draw [plineb] (0.5,3.5) -- (3.5,3.5);
\draw [plineb] (0.5,2.5) -- (3.5,2.5);
\draw [plineb] (0.5,1.5) -- (3.5,1.5);
\draw [plineb] (0.5,0.5) -- (1.5,0.5);
\draw [plineb] (0.5,3.5) -- (0.5,0.5);
\draw [plineb] (1.5,3.5) -- (1.5,0.5);
\draw [plineb] (2.5,3.5) -- (2.5,1.5);
\draw [plineb] (3.5,3.5) -- (3.5,1.5);

\node at (1,3) {$\scriptstyle 4$};
\node at (1,2) {$\scriptstyle 7$};
\node at (1,1) {$\scriptstyle 1$};

\node at (2,3) {$\scriptstyle 6$};
\node at (2,2) {$\scriptstyle 3$};

\node at (3,3) {$\scriptstyle 2$};
\node at (3,2) {$\scriptstyle 5$};

\draw [very thick] (1.5,3.5) rectangle (2.5,2.5);

\end{tikzpicture}
};

\node (d) at (9,2.5) {
\begin{tikzpicture}[scale=0.6]

\draw [plineb] (0.5,3.5) -- (3.5,3.5);
\draw [plineb] (0.5,2.5) -- (3.5,2.5);
\draw [plineb] (0.5,1.5) -- (3.5,1.5);
\draw [plineb] (0.5,0.5) -- (1.5,0.5);
\draw [plineb] (0.5,3.5) -- (0.5,0.5);
\draw [plineb] (1.5,3.5) -- (1.5,0.5);
\draw [plineb] (2.5,3.5) -- (2.5,1.5);
\draw [plineb] (3.5,3.5) -- (3.5,1.5);

\node at (1,3) {$\scriptstyle 4$};
\node at (1,2) {$\scriptstyle 7$};
\node at (1,1) {$\scriptstyle 1$};

\node at (2,3) {$\scriptstyle 2$};
\node at (2,2) {$\scriptstyle 3$};

\node at (3,3) {$\scriptstyle 6$};
\node at (3,2) {$\scriptstyle 5$};

\draw [very thick] (2.5,3.5) rectangle (3.5,2.5);

\end{tikzpicture}
};

\node (e) at (0,0) {
\begin{tikzpicture}[scale=0.6]

\draw [plineb] (0.5,3.5) -- (3.5,3.5);
\draw [plineb] (0.5,2.5) -- (3.5,2.5);
\draw [plineb] (0.5,1.5) -- (3.5,1.5);
\draw [plineb] (0.5,0.5) -- (1.5,0.5);
\draw [plineb] (0.5,3.5) -- (0.5,0.5);
\draw [plineb] (1.5,3.5) -- (1.5,0.5);
\draw [plineb] (2.5,3.5) -- (2.5,1.5);
\draw [plineb] (3.5,3.5) -- (3.5,1.5);

\node at (1,3) {$\scriptstyle 4$};
\node at (1,2) {$\scriptstyle 7$};
\node at (1,1) {$\scriptstyle 1$};

\node at (2,3) {$\scriptstyle 2$};
\node at (2,2) {$\scriptstyle 3$};

\node at (3,3) {$\scriptstyle 5$};
\node at (3,2) {$\scriptstyle 6$};

\draw [very thick] (0.5,2.5) rectangle (1.5,1.5);

\end{tikzpicture}
};

\node (f) at (3,0) {
\begin{tikzpicture}[scale=0.6]

\draw [plineb] (0.5,3.5) -- (3.5,3.5);
\draw [plineb] (0.5,2.5) -- (3.5,2.5);
\draw [plineb] (0.5,1.5) -- (3.5,1.5);
\draw [plineb] (0.5,0.5) -- (1.5,0.5);
\draw [plineb] (0.5,3.5) -- (0.5,0.5);
\draw [plineb] (1.5,3.5) -- (1.5,0.5);
\draw [plineb] (2.5,3.5) -- (2.5,1.5);
\draw [plineb] (3.5,3.5) -- (3.5,1.5);

\node at (1,3) {$\scriptstyle 4$};
\node at (1,2) {$\scriptstyle 1$};
\node at (1,1) {$\scriptstyle 7$};

\node at (2,3) {$\scriptstyle 2$};
\node at (2,2) {$\scriptstyle 3$};

\node at (3,3) {$\scriptstyle 5$};
\node at (3,2) {$\scriptstyle 6$};

\draw [very thick] (0.5,3.5) rectangle (1.5,2.5);

\end{tikzpicture}
};

\node (g) at (6,0) {
\begin{tikzpicture}[scale=0.6]

\draw [plineb] (0.5,3.5) -- (3.5,3.5);
\draw [plineb] (0.5,2.5) -- (3.5,2.5);
\draw [plineb] (0.5,1.5) -- (3.5,1.5);
\draw [plineb] (0.5,0.5) -- (1.5,0.5);
\draw [plineb] (0.5,3.5) -- (0.5,0.5);
\draw [plineb] (1.5,3.5) -- (1.5,0.5);
\draw [plineb] (2.5,3.5) -- (2.5,1.5);
\draw [plineb] (3.5,3.5) -- (3.5,1.5);

\node at (1,3) {$\scriptstyle 1$};
\node at (1,2) {$\scriptstyle 4$};
\node at (1,1) {$\scriptstyle 7$};

\node at (2,3) {$\scriptstyle 2$};
\node at (2,2) {$\scriptstyle 3$};

\node at (3,3) {$\scriptstyle 5$};
\node at (3,2) {$\scriptstyle 6$};

\draw [very thick] (0.5,2.5) rectangle (1.5,1.5);

\end{tikzpicture}
};

\node (h) at (9,0) {
\begin{tikzpicture}[scale=0.6]

\draw [plineb] (0.5,3.5) -- (3.5,3.5);
\draw [plineb] (0.5,2.5) -- (3.5,2.5);
\draw [plineb] (0.5,1.5) -- (3.5,1.5);
\draw [plineb] (0.5,0.5) -- (1.5,0.5);
\draw [plineb] (0.5,3.5) -- (0.5,0.5);
\draw [plineb] (1.5,3.5) -- (1.5,0.5);
\draw [plineb] (2.5,3.5) -- (2.5,1.5);
\draw [plineb] (3.5,3.5) -- (3.5,1.5);

\draw [very thick] (0.5,3.5)--(3.5,3.5)--(3.5,1.5)--(1.5,1.5)--(1.5,0.5)--(0.5,0.5)--cycle;

\node at (1,3) {$\scriptstyle 1$};
\node at (1,2) {$\scriptstyle 3$};
\node at (1,1) {$\scriptstyle 7$};

\node at (2,3) {$\scriptstyle 2$};
\node at (2,2) {$\scriptstyle 4$};

\node at (3,3) {$\scriptstyle 5$};
\node at (3,2) {$\scriptstyle 6$};

\end{tikzpicture}
};

\node at (-1.5,2.5) {$T$};
\draw [post,very thick] (1.1,2.5) -- (1.9,2.5);
\draw [post,very thick] (4.1,2.5) -- (4.9,2.5);
\draw [post,very thick] (7.1,2.5) -- (7.9,2.5);
\draw [post,very thick] (10.1,2.5) -- (10.9,2.5);

\draw [post,very thick] (-1.9,0) -- (-1.1,0);
\draw [post,very thick] (1.1,0) -- (1.9,0);
\draw [post,very thick] (4.1,0) -- (4.9,0);
\draw [post,very thick] (7.1,0) -- (7.9,0);
\node at (10.5,0) {$W$};

\end{tikzpicture}
}

\def\FIGorder{
\begin{tikzpicture}[plineb/.style={draw=gray,thin},plineg/.style={draw=lightgray}]
\node (a) at (0,0) {
\begin{tikzpicture}[scale=0.6,post/.style={->,shorten >=1pt,>=stealth',semithick},pre/.style={<-,shorten >=1pt,>=stealth',semithick}]
\draw [plineb] (0.5,3.5) -- (3.5,3.5);
\draw [plineb] (0.5,2.5) -- (3.5,2.5);
\draw [plineb] (0.5,1.5) -- (3.5,1.5);
\draw [plineb] (0.5,0.5) -- (1.5,0.5);

\draw [plineb] (0.5,3.5) -- (0.5,0.5);
\draw [plineb] (1.5,3.5) -- (1.5,0.5);
\draw [plineb] (2.5,3.5) -- (2.5,1.5);
\draw [plineb] (3.5,3.5) -- (3.5,1.5);

\node at (1,3) {$\scriptstyle 1$};
\node at (1,2) {$\scriptstyle 2$};
\node at (1,1) {$\scriptstyle 3$};
\draw [very thick] (0.5,3.5) rectangle (1.5,0.5);

\node at (2,3) {$\scriptstyle 4$};
\node at (2,2) {$\scriptstyle 5$};
\draw [very thick] (1.5,3.5) rectangle (2.5,1.5);

\node at (3,3) {$\scriptstyle 6$};
\node at (3,2) {$\scriptstyle 7$};
\draw [very thick] (2.5,3.5) rectangle (3.5,1.5);

\end{tikzpicture}
};

\node (b) at (5,0) {
\begin{tikzpicture}[scale=0.6,post/.style={->,shorten >=1pt,>=stealth',semithick},pre/.style={<-,shorten >=1pt,>=stealth',semithick}]
\draw [plineb] (0.5,3.5) -- (3.5,3.5);
\draw [plineb] (0.5,2.5) -- (3.5,2.5);
\draw [plineb] (0.5,1.5) -- (3.5,1.5);
\draw [plineb] (0.5,0.5) -- (1.5,0.5);

\draw [plineb] (0.5,3.5) -- (0.5,0.5);
\draw [plineb] (1.5,3.5) -- (1.5,0.5);
\draw [plineb] (2.5,3.5) -- (2.5,1.5);
\draw [plineb] (3.5,3.5) -- (3.5,1.5);

\node at (1,3) {$\scriptstyle 1$};
\node at (2,3) {$\scriptstyle 2$};
\node at (3,3) {$\scriptstyle 3$};
\draw [very thick] (0.5,3.5) rectangle (3.5,2.5);

\node at (1,2) {$\scriptstyle 4$};
\node at (2,2) {$\scriptstyle 5$};
\node at (3,2) {$\scriptstyle 6$};
\draw [very thick] (0.5,2.5) rectangle (3.5,1.5);

\node at (1,1) {$\scriptstyle 7$};
\draw [very thick] (0.5,1.5) rectangle (1.5,0.5);

\end{tikzpicture}
};

\end{tikzpicture}
}

\def\FIGotherorder{
\begin{tikzpicture}[plineb/.style={draw=gray,thin},plineg/.style={draw=lightgray}]
\node at (0,-0) {
\begin{tikzpicture}[scale=0.6,post/.style={->,shorten >=1pt,>=stealth',semithick},pre/.style={<-,shorten >=1pt,>=stealth',semithick}]
\draw [plineb] (0.5,8.5) -- (16.5,8.5);
\draw [plineb] (0.5,7.5) -- (16.5,7.5);
\draw [plineb] (0.5,6.5) -- (15.5,6.5);
\draw [plineb] (0.5,5.5) -- (14.5,5.5);
\draw [plineb] (0.5,4.5) -- (14.5,4.5);
\draw [plineb] (0.5,3.5) -- (14.5,3.5);
\draw [plineb] (0.5,2.5) -- (14.5,2.5);
\draw [plineb] (0.5,1.5) -- (13.5,1.5);
\draw [plineb] (0.5,0.5) -- (13.5,0.5);
\draw [plineb] (0.5,-0.5) -- (12.5,-0.5);
\draw [plineb] (0.5,-1.5) -- (5.5,-1.5);

\draw [plineb] (0.5,8.5) -- (0.5,-1.5);
\draw [plineb] (1.5,8.5) -- (1.5,-1.5);
\draw [plineb] (2.5,8.5) -- (2.5,-1.5);
\draw [plineb] (3.5,8.5) -- (3.5,-1.5);
\draw [plineb] (4.5,8.5) -- (4.5,-1.5);
\draw [plineb] (5.5,8.5) -- (5.5,-1.5);
\draw [plineb] (6.5,8.5) -- (6.5,-0.5);
\draw [plineb] (7.5,8.5) -- (7.5,-0.5);
\draw [plineb] (8.5,8.5) -- (8.5,-0.5);
\draw [plineb] (9.5,8.5) -- (9.5,-0.5);
\draw [plineb] (10.5,8.5) -- (10.5,-0.5);
\draw [plineb] (11.5,8.5) -- (11.5,-0.5);
\draw [plineb] (12.5,8.5) -- (12.5,-0.5);
\draw [plineb] (13.5,8.5) -- (13.5,0.5);
\draw [plineb] (14.5,8.5) -- (14.5,2.5);
\draw [plineb] (15.5,8.5) -- (15.5,6.5);
\draw [plineb] (16.5,8.5) -- (16.5,7.5);

\node at (1,8) {$\scriptstyle 1$};
\node at (2,8) {$\scriptstyle 2$};
\node at (3,8) {$\scriptstyle 3$};
\node at (4,8) {$\scriptstyle 4$};
\node at (5,8) {$\scriptstyle 5$};
\node at (6,8) {$\scriptstyle 6$};
\node at (7,8) {$\scriptstyle 7$};
\node at (8,8) {$\scriptstyle 8$};
\node at (9,8) {$\scriptstyle 9$};
\node at (10,8) {$\scriptstyle 10$};
\node at (11,8) {$\scriptstyle 11$};
\node at (12,8) {$\scriptstyle 12$};
\node at (13,8) {$\scriptstyle 13$};
\node at (14,8) {$\scriptstyle 14$};
\node at (15,8) {$\scriptstyle 15$};
\node at (16,8) {$\scriptstyle 16$};
\draw [very thick] (0.5,7.5) rectangle (16.5,8.5);

\node at (1,7) {$\scriptstyle 17$};
\node at (1,6) {$\scriptstyle 18$};
\node at (1,5) {$\scriptstyle 19$};
\node at (1,4) {$\scriptstyle 20$};
\node at (1,3) {$\scriptstyle 21$};
\node at (1,2) {$\scriptstyle 22$};
\node at (1,1) {$\scriptstyle 23$};
\node at (1,0) {$\scriptstyle 24$};
\node at (1,-1) {$\scriptstyle 25$};
\draw [very thick] (0.5,-1.5) rectangle (1.5,7.5);

\node at (2,7) {$\scriptstyle 26$};
\node at (2,6) {$\scriptstyle 27$};
\node at (2,5) {$\scriptstyle 28$};
\node at (2,4) {$\scriptstyle 29$};
\node at (2,3) {$\scriptstyle 30$};
\node at (2,2) {$\scriptstyle 31$};
\node at (2,1) {$\scriptstyle 32$};
\node at (2,0) {$\scriptstyle 33$};
\node at (2,-1) {$\scriptstyle 34$};
\draw [very thick] (1.5,-1.5) rectangle (2.5,7.5);

\node at (3,7) {$\scriptstyle 35$};
\node at (3,6) {$\scriptstyle 36$};
\node at (3,5) {$\scriptstyle 37$};
\node at (3,4) {$\scriptstyle 38$};
\node at (3,3) {$\scriptstyle 39$};
\node at (3,2) {$\scriptstyle 40$};
\node at (3,1) {$\scriptstyle 41$};
\node at (3,0) {$\scriptstyle 42$};
\node at (3,-1) {$\scriptstyle 43$};
\draw [very thick] (2.5,-1.5) rectangle (3.5,7.5);

\node at (4,7) {$\scriptstyle 44$};
\node at (5,7) {$\scriptstyle 45$};
\node at (6,7) {$\scriptstyle 46$};
\node at (7,7) {$\scriptstyle 47$};
\node at (8,7) {$\scriptstyle 48$};
\node at (9,7) {$\scriptstyle 49$};
\node at (10,7) {$\scriptstyle 50$};
\node at (11,7) {$\scriptstyle 51$};
\node at (12,7) {$\scriptstyle 52$};
\node at (13,7) {$\scriptstyle 53$};
\node at (14,7) {$\scriptstyle 54$};
\node at (15,7) {$\scriptstyle 55$};
\draw [very thick] (3.5,6.5) rectangle (15.5,7.5);

\node at (4,6) {$\scriptstyle 56$};
\node at (5,6) {$\scriptstyle 57$};
\node at (6,6) {$\scriptstyle 58$};
\node at (7,6) {$\scriptstyle 59$};
\node at (8,6) {$\scriptstyle 60$};
\node at (9,6) {$\scriptstyle 61$};
\node at (10,6) {$\scriptstyle 62$};
\node at (11,6) {$\scriptstyle 63$};
\node at (12,6) {$\scriptstyle 64$};
\node at (13,6) {$\scriptstyle 65$};
\node at (14,6) {$\scriptstyle 66$};
\draw [very thick] (3.5,5.5) rectangle (14.5,6.5);

\node at (4,5) {$\scriptstyle 67$};
\node at (4,4) {$\scriptstyle 68$};
\node at (4,3) {$\scriptstyle 69$};
\node at (4,2) {$\scriptstyle 70$};
\node at (4,1) {$\scriptstyle 71$};
\node at (4,0) {$\scriptstyle 72$};
\node at (4,-1) {$\scriptstyle 73$};
\draw [very thick] (3.5,-1.5) rectangle (4.5,5.5);

\node at (5,5) {$\scriptstyle 74$};
\node at (6,5) {$\scriptstyle 75$};
\node at (7,5) {$\scriptstyle 76$};
\node at (8,5) {$\scriptstyle 77$};
\node at (9,5) {$\scriptstyle 78$};
\node at (10,5) {$\scriptstyle 79$};
\node at (11,5) {$\scriptstyle 80$};
\node at (12,5) {$\scriptstyle 81$};
\node at (13,5) {$\scriptstyle 82$};
\node at (14,5) {$\scriptstyle 83$};
\draw [very thick] (4.5,4.5) rectangle (14.5,5.5);

\node at (5,4) {$\scriptstyle 84$};
\node at (6,4) {$\scriptstyle 85$};
\node at (7,4) {$\scriptstyle 86$};
\node at (8,4) {$\scriptstyle 87$};
\node at (9,4) {$\scriptstyle 88$};
\node at (10,4) {$\scriptstyle 89$};
\node at (11,4) {$\scriptstyle 90$};
\node at (12,4) {$\scriptstyle 91$};
\node at (13,4) {$\scriptstyle 92$};
\node at (14,4) {$\scriptstyle 93$};
\draw [very thick] (4.5,3.5) rectangle (14.5,4.5);

\node at (5,3) {$\scriptstyle 94$};
\node at (6,3) {$\scriptstyle 95$};
\node at (7,3) {$\scriptstyle 96$};
\node at (8,3) {$\scriptstyle 97$};
\node at (9,3) {$\scriptstyle 98$};
\node at (10,3) {$\scriptstyle 99$};
\node at (11,3) {$\scriptstyle 100$};
\node at (12,3) {$\scriptstyle 101$};
\node at (13,3) {$\scriptstyle 102$};
\node at (14,3) {$\scriptstyle 103$};
\draw [very thick] (4.5,2.5) rectangle (14.5,3.5);

\node at (5,2) {$\scriptstyle 104$};
\node at (5,1) {$\scriptstyle 105$};
\node at (5,0) {$\scriptstyle 106$};
\node at (5,-1) {$\scriptstyle 107$};
\draw [very thick] (4.5,-1.5) rectangle (5.5,2.5);

\node at (6,2) {$\scriptstyle 108$};
\node at (6,1) {$\scriptstyle 109$};
\node at (6,0) {$\scriptstyle 110$};
\draw [very thick] (5.5,-0.5) rectangle (6.5,2.5);

\node at (7,2) {$\scriptstyle 111$};
\node at (7,1) {$\scriptstyle 112$};
\node at (7,0) {$\scriptstyle 113$};
\draw [very thick] (6.5,-0.5) rectangle (7.5,2.5);

\node at (8,2) {$\scriptstyle 114$};
\node at (9,2) {$\scriptstyle 115$};
\node at (10,2) {$\scriptstyle 116$};
\node at (11,2) {$\scriptstyle 117$};
\node at (12,2) {$\scriptstyle 118$};
\node at (13,2) {$\scriptstyle 119$};
\draw [very thick] (7.5,1.5) rectangle (13.5,2.5);

\node at (8,1) {$\scriptstyle 120$};
\node at (9,1) {$\scriptstyle 121$};
\node at (10,1) {$\scriptstyle 122$};
\node at (11,1) {$\scriptstyle 123$};
\node at (12,1) {$\scriptstyle 124$};
\node at (13,1) {$\scriptstyle 125$};
\draw [very thick] (7.5,0.5) rectangle (13.5,1.5);

\node at (8,0) {$\scriptstyle 126$};
\node at (9,0) {$\scriptstyle 127$};
\node at (10,0) {$\scriptstyle 128$};
\node at (11,0) {$\scriptstyle 129$};
\node at (12,0) {$\scriptstyle 130$};
\draw [very thick] (7.5,-0.5) rectangle (12.5,0.5);

\end{tikzpicture}
};

\end{tikzpicture}
}

\def\FIGexchangeOne{
\begin{tikzpicture}[post/.style={->,shorten >=1pt,>=stealth',semithick},plineb/.style={draw=black},plineg/.style={draw=lightgray}]
\node (a) at (0,0) {
\begin{tikzpicture}[scale=0.7]
\fill [fill=gray] (3.5,5.5) -- (4.5,5.5) -- (4.5,8.5) -- (8.5,8.5) -- (8.5,7.5) -- (7.5,7.5) -- (7.5,6.5) -- (6.5,6.5) -- (6.5,4.5) -- (5.5,4.5) -- (5.5,3.5) -- (4.5,3.5) -- (4.5,2.5) -- (3.5,2.5) -- (3.5,5.5); 

\draw [plineb] (0.5,8.5) -- (8.5,8.5);
\draw [plineb] (0.5,7.5) -- (8.5,7.5);
\draw [plineb] (0.5,6.5) -- (7.5,6.5);
\draw [plineb] (0.5,5.5) -- (6.5,5.5);
\draw [plineb] (0.5,4.5) -- (6.5,4.5);
\draw [plineb] (0.5,3.5) -- (5.5,3.5);
\draw [plineb] (0.5,2.5) -- (4.5,2.5);

\draw [plineb] (0.5,8.5) -- (0.5,2.5);
\draw [plineb] (1.5,8.5) -- (1.5,2.5);
\draw [plineb] (2.5,8.5) -- (2.5,2.5);
\draw [plineb] (3.5,8.5) -- (3.5,2.5);
\draw [plineb] (4.5,8.5) -- (4.5,2.5);
\draw [plineb] (5.5,8.5) -- (5.5,3.5);
\draw [plineb] (6.5,8.5) -- (6.5,4.5);
\draw [plineb] (7.5,8.5) -- (7.5,6.5);
\draw [plineb] (8.5,8.5) -- (8.5,7.5);

\node at (1.75,7.25) {$\scriptstyle x_i$};
\node at (2,6.8) {$\scriptstyle b$};

\node at (3.75,6.25) {$\scriptstyle x_j$};
\node at (4,5.8) {$\scriptstyle b+1$};
\end{tikzpicture}
};

\node (b) at (8.2,0) {
\begin{tikzpicture}[scale=0.7]
\fill [fill=gray] (3.5,5.5) -- (4.5,5.5) -- (4.5,8.5) -- (8.5,8.5) -- (8.5,7.5) -- (7.5,7.5) -- (7.5,6.5) -- (6.5,6.5) -- (6.5,4.5) -- (5.5,4.5) -- (5.5,3.5) -- (4.5,3.5) -- (4.5,2.5) -- (3.5,2.5) -- (3.5,5.5); 

\draw [plineb] (0.5,8.5) -- (8.5,8.5);
\draw [plineb] (0.5,7.5) -- (8.5,7.5);
\draw [plineb] (0.5,6.5) -- (7.5,6.5);
\draw [plineb] (0.5,5.5) -- (6.5,5.5);
\draw [plineb] (0.5,4.5) -- (6.5,4.5);
\draw [plineb] (0.5,3.5) -- (5.5,3.5);
\draw [plineb] (0.5,2.5) -- (4.5,2.5);

\draw [plineb] (0.5,8.5) -- (0.5,2.5);
\draw [plineb] (1.5,8.5) -- (1.5,2.5);
\draw [plineb] (2.5,8.5) -- (2.5,2.5);
\draw [plineb] (3.5,8.5) -- (3.5,2.5);
\draw [plineb] (4.5,8.5) -- (4.5,2.5);
\draw [plineb] (5.5,8.5) -- (5.5,3.5);
\draw [plineb] (6.5,8.5) -- (6.5,4.5);
\draw [plineb] (7.5,8.5) -- (7.5,6.5);
\draw [plineb] (8.5,8.5) -- (8.5,7.5);

\node at (1.75,7.25) {$\scriptstyle x_i$};
\node at (2,6.8) {$\scriptstyle b+1$};

\node at (3.75,6.25) {$\scriptstyle x_j$};
\node at (4,5.8) {$\scriptstyle b$};
\end{tikzpicture}
};

\node at (2,-1) {$T_k$};
\node at (10.2,-1) {$T_k^*$};

\end{tikzpicture}
}

\def\FIGexchangeTwo{
\begin{tikzpicture}[post/.style={->,shorten >=1pt,>=stealth',semithick},plineb/.style={draw=black},plineg/.style={draw=lightgray}]
\node (a) at (0,0) {
\begin{tikzpicture}[scale=0.7]
\fill [fill=lightgray] (3.5,5.5) -- (4.5,5.5) -- (4.5,8.5) -- (8.5,8.5) -- (8.5,7.5) -- (7.5,7.5) -- (7.5,6.5) -- (6.5,6.5) -- (6.5,4.5) -- (5.5,4.5) -- (5.5,3.5) -- (4.5,3.5) -- (4.5,2.5) -- (3.5,2.5) -- (3.5,5.5); 
\fill [fill=gray] (3.5,5.5) rectangle (4.5,6.5);

\draw [plineb] (0.5,8.5) -- (8.5,8.5);
\draw [plineb] (0.5,7.5) -- (8.5,7.5);
\draw [plineb] (0.5,6.5) -- (7.5,6.5);
\draw [plineb] (0.5,5.5) -- (6.5,5.5);
\draw [plineb] (0.5,4.5) -- (6.5,4.5);
\draw [plineb] (0.5,3.5) -- (5.5,3.5);
\draw [plineb] (0.5,2.5) -- (4.5,2.5);

\draw [plineb] (0.5,8.5) -- (0.5,2.5);
\draw [plineb] (1.5,8.5) -- (1.5,2.5);
\draw [plineb] (2.5,8.5) -- (2.5,2.5);
\draw [plineb] (3.5,8.5) -- (3.5,2.5);
\draw [plineb] (4.5,8.5) -- (4.5,2.5);
\draw [plineb] (5.5,8.5) -- (5.5,3.5);
\draw [plineb] (6.5,8.5) -- (6.5,4.5);
\draw [plineb] (7.5,8.5) -- (7.5,6.5);
\draw [plineb] (8.5,8.5) -- (8.5,7.5);

\draw [post,ultra thick,white] (4,6) -- (5,6) -- (5,5) -- (6,5);

\node at (1.75,7.25) {$\scriptstyle x_i$};
\node at (2,6.8) {$\scriptstyle b$};

\node at (3.75,6.25) {$\scriptstyle x_j$};
\node at (6,5) {$\scriptstyle b+1$};
\end{tikzpicture}
};

\node (b) at (8.2,0) {
\begin{tikzpicture}[scale=0.7]
\fill [fill=lightgray] (3.5,5.5) -- (4.5,5.5) -- (4.5,8.5) -- (8.5,8.5) -- (8.5,7.5) -- (7.5,7.5) -- (7.5,6.5) -- (6.5,6.5) -- (6.5,4.5) -- (5.5,4.5) -- (5.5,3.5) -- (4.5,3.5) -- (4.5,2.5) -- (3.5,2.5) -- (3.5,5.5); 
\fill [fill=gray] (3.5,5.5) rectangle (4.5,6.5);

\draw [plineb] (0.5,8.5) -- (8.5,8.5);
\draw [plineb] (0.5,7.5) -- (8.5,7.5);
\draw [plineb] (0.5,6.5) -- (7.5,6.5);
\draw [plineb] (0.5,5.5) -- (6.5,5.5);
\draw [plineb] (0.5,4.5) -- (6.5,4.5);
\draw [plineb] (0.5,3.5) -- (5.5,3.5);
\draw [plineb] (0.5,2.5) -- (4.5,2.5);

\draw [plineb] (0.5,8.5) -- (0.5,2.5);
\draw [plineb] (1.5,8.5) -- (1.5,2.5);
\draw [plineb] (2.5,8.5) -- (2.5,2.5);
\draw [plineb] (3.5,8.5) -- (3.5,2.5);
\draw [plineb] (4.5,8.5) -- (4.5,2.5);
\draw [plineb] (5.5,8.5) -- (5.5,3.5);
\draw [plineb] (6.5,8.5) -- (6.5,4.5);
\draw [plineb] (7.5,8.5) -- (7.5,6.5);
\draw [plineb] (8.5,8.5) -- (8.5,7.5);

\draw [post,ultra thick,white] (4,6) -- (5,6) -- (5,5) -- (6,5);

\node at (1.75,7.25) {$\scriptstyle x_i$};
\node at (2,6.8) {$\scriptstyle b+1$};

\node at (3.75,6.25) {$\scriptstyle x_j$};
\node at (6,5) {$\scriptstyle b$};
\end{tikzpicture}
};

\node at (2,-1) {$T_k$};
\node at (10.2,-1) {$T_k^*$};

\end{tikzpicture}
}

\def\FIGexchangeThree{
\begin{tikzpicture}[post/.style={->,shorten >=1pt,>=stealth',semithick},plineb/.style={draw=black},plineg/.style={draw=lightgray}]
\node (a) at (0,0) {
\begin{tikzpicture}[scale=0.7]
\fill [fill=lightgray] (3.5,6.5) -- (4.5,6.5) -- (4.5,8.5) -- (8.5,8.5) -- (8.5,7.5) -- (7.5,7.5) -- (7.5,6.5) -- (6.5,6.5) -- (6.5,4.5) -- (5.5,4.5) -- (5.5,3.5) -- (4.5,3.5) -- (4.5,2.5) -- (3.5,2.5) -- (3.5,6.5); 
\fill [fill=gray] (1.5,6.5) -- (2.5,6.5) -- (2.5,8.5) -- (4.5,8.5) -- (4.5,6.5) -- (3.5,6.5) -- (3.5,2.5) -- (1.5,2.5) -- (1.5,6.5);

\draw [plineb] (0.5,8.5) -- (8.5,8.5);
\draw [plineb] (0.5,7.5) -- (8.5,7.5);
\draw [plineb] (0.5,6.5) -- (7.5,6.5);
\draw [plineb] (0.5,5.5) -- (6.5,5.5);
\draw [plineb] (0.5,4.5) -- (6.5,4.5);
\draw [plineb] (0.5,3.5) -- (5.5,3.5);
\draw [plineb] (0.5,2.5) -- (4.5,2.5);

\draw [plineb] (0.5,8.5) -- (0.5,2.5);
\draw [plineb] (1.5,8.5) -- (1.5,2.5);
\draw [plineb] (2.5,8.5) -- (2.5,2.5);
\draw [plineb] (3.5,8.5) -- (3.5,2.5);
\draw [plineb] (4.5,8.5) -- (4.5,2.5);
\draw [plineb] (5.5,8.5) -- (5.5,3.5);
\draw [plineb] (6.5,8.5) -- (6.5,4.5);
\draw [plineb] (7.5,8.5) -- (7.5,6.5);
\draw [plineb] (8.5,8.5) -- (8.5,7.5);

\draw [post,ultra thick,white] (6,5) -- (6,6);
\draw [post,ultra thick,white] (6,6) -- (5,6);

\node at (1.75,7.25) {$\scriptstyle x_i$};
\node at (2,6.8) {$\scriptstyle b$};

\node at (3.75,6.25) {$\scriptstyle x_j$};
\node at (5,6) {$\scriptstyle b+1$};
\end{tikzpicture}
};

\node (b) at (8.2,0) {
\begin{tikzpicture}[scale=0.7]
\fill [fill=lightgray] (3.5,6.5) -- (4.5,6.5) -- (4.5,8.5) -- (8.5,8.5) -- (8.5,7.5) -- (7.5,7.5) -- (7.5,6.5) -- (6.5,6.5) -- (6.5,4.5) -- (5.5,4.5) -- (5.5,3.5) -- (4.5,3.5) -- (4.5,2.5) -- (3.5,2.5) -- (3.5,6.5); 
\fill [fill=gray] (1.5,6.5) -- (2.5,6.5) -- (2.5,8.5) -- (4.5,8.5) -- (4.5,6.5) -- (3.5,6.5) -- (3.5,2.5) -- (1.5,2.5) -- (1.5,6.5);

\draw [plineb] (0.5,8.5) -- (8.5,8.5);
\draw [plineb] (0.5,7.5) -- (8.5,7.5);
\draw [plineb] (0.5,6.5) -- (7.5,6.5);
\draw [plineb] (0.5,5.5) -- (6.5,5.5);
\draw [plineb] (0.5,4.5) -- (6.5,4.5);
\draw [plineb] (0.5,3.5) -- (5.5,3.5);
\draw [plineb] (0.5,2.5) -- (4.5,2.5);

\draw [plineb] (0.5,8.5) -- (0.5,2.5);
\draw [plineb] (1.5,8.5) -- (1.5,2.5);
\draw [plineb] (2.5,8.5) -- (2.5,2.5);
\draw [plineb] (3.5,8.5) -- (3.5,2.5);
\draw [plineb] (4.5,8.5) -- (4.5,2.5);
\draw [plineb] (5.5,8.5) -- (5.5,3.5);
\draw [plineb] (6.5,8.5) -- (6.5,4.5);
\draw [plineb] (7.5,8.5) -- (7.5,6.5);
\draw [plineb] (8.5,8.5) -- (8.5,7.5);

\draw [post,ultra thick,white] (6,5) -- (6,6);
\draw [post,ultra thick,white] (6,6) -- (5,6);

\node at (1.75,7.25) {$\scriptstyle x_i$};
\node at (2,6.8) {$\scriptstyle b+1$};

\node at (3.75,6.25) {$\scriptstyle x_j$};
\node at (5,6) {$\scriptstyle b$};
\end{tikzpicture}
};

\node at (2,-1) {$T_k$};
\node at (10.2,-1) {$T_k^*$};

\end{tikzpicture}
}

\def\FIGexchangeFour{
\begin{tikzpicture}[post/.style={->,shorten >=1pt,>=stealth',semithick},plineb/.style={draw=black},plineg/.style={draw=lightgray}]
\node (a) at (0,0) {
\begin{tikzpicture}[scale=0.7]
\fill [fill=lightgray] (3.5,6.5) -- (4.5,6.5) -- (4.5,8.5) -- (8.5,8.5) -- (8.5,7.5) -- (7.5,7.5) -- (7.5,6.5) -- (6.5,6.5) -- (6.5,4.5) -- (5.5,4.5) -- (5.5,3.5) -- (4.5,3.5) -- (4.5,2.5) -- (3.5,2.5) -- (3.5,6.5); 
\fill [fill=lightgray] (1.5,6.5) -- (2.5,6.5) -- (2.5,8.5) -- (4.5,8.5) -- (4.5,6.5) -- (3.5,6.5) -- (3.5,2.5) -- (1.5,2.5) -- (1.5,6.5);
\fill [fill=gray] (1.5,6.5) rectangle (2.5,7.5);

\draw [plineb] (0.5,8.5) -- (8.5,8.5);
\draw [plineb] (0.5,7.5) -- (8.5,7.5);
\draw [plineb] (0.5,6.5) -- (7.5,6.5);
\draw [plineb] (0.5,5.5) -- (6.5,5.5);
\draw [plineb] (0.5,4.5) -- (6.5,4.5);
\draw [plineb] (0.5,3.5) -- (5.5,3.5);
\draw [plineb] (0.5,2.5) -- (4.5,2.5);

\draw [plineb] (0.5,8.5) -- (0.5,2.5);
\draw [plineb] (1.5,8.5) -- (1.5,2.5);
\draw [plineb] (2.5,8.5) -- (2.5,2.5);
\draw [plineb] (3.5,8.5) -- (3.5,2.5);
\draw [plineb] (4.5,8.5) -- (4.5,2.5);
\draw [plineb] (5.5,8.5) -- (5.5,3.5);
\draw [plineb] (6.5,8.5) -- (6.5,4.5);
\draw [plineb] (7.5,8.5) -- (7.5,6.5);
\draw [plineb] (8.5,8.5) -- (8.5,7.5);

\draw [post,ultra thick,white] (2,7) -- (2,6) -- (3,6) -- (3,5);

\node at (1.75,7.25) {$\scriptstyle x_i$};
\node at (3,5) {$\scriptstyle b$};

\node at (3.75,6.25) {$\scriptstyle x_j$};
\node at (5,6) {$\scriptstyle b+1$};
\end{tikzpicture}
};

\node (b) at (8.2,0) {
\begin{tikzpicture}[scale=0.7]
\fill [fill=lightgray] (3.5,6.5) -- (4.5,6.5) -- (4.5,8.5) -- (8.5,8.5) -- (8.5,7.5) -- (7.5,7.5) -- (7.5,6.5) -- (6.5,6.5) -- (6.5,4.5) -- (5.5,4.5) -- (5.5,3.5) -- (4.5,3.5) -- (4.5,2.5) -- (3.5,2.5) -- (3.5,6.5); 
\fill [fill=lightgray] (1.5,6.5) -- (2.5,6.5) -- (2.5,8.5) -- (4.5,8.5) -- (4.5,6.5) -- (3.5,6.5) -- (3.5,2.5) -- (1.5,2.5) -- (1.5,6.5);
\fill [fill=gray] (1.5,6.5) rectangle (2.5,7.5);

\draw [plineb] (0.5,8.5) -- (8.5,8.5);
\draw [plineb] (0.5,7.5) -- (8.5,7.5);
\draw [plineb] (0.5,6.5) -- (7.5,6.5);
\draw [plineb] (0.5,5.5) -- (6.5,5.5);
\draw [plineb] (0.5,4.5) -- (6.5,4.5);
\draw [plineb] (0.5,3.5) -- (5.5,3.5);
\draw [plineb] (0.5,2.5) -- (4.5,2.5);

\draw [plineb] (0.5,8.5) -- (0.5,2.5);
\draw [plineb] (1.5,8.5) -- (1.5,2.5);
\draw [plineb] (2.5,8.5) -- (2.5,2.5);
\draw [plineb] (3.5,8.5) -- (3.5,2.5);
\draw [plineb] (4.5,8.5) -- (4.5,2.5);
\draw [plineb] (5.5,8.5) -- (5.5,3.5);
\draw [plineb] (6.5,8.5) -- (6.5,4.5);
\draw [plineb] (7.5,8.5) -- (7.5,6.5);
\draw [plineb] (8.5,8.5) -- (8.5,7.5);

\draw [post,ultra thick,white] (2,7) -- (2,6) -- (3,6) -- (3,5);

\node at (1.75,7.25) {$\scriptstyle x_i$};
\node at (3,5) {$\scriptstyle b+1$};

\node at (3.75,6.25) {$\scriptstyle x_j$};
\node at (5,6) {$\scriptstyle b$};
\end{tikzpicture}
};

\node at (2,-1) {$T_k$};
\node at (10.2,-1) {$T_k^*$};

\end{tikzpicture}
}

\def\FIGexchangeFive{
\begin{tikzpicture}[post/.style={->,shorten >=1pt,>=stealth',semithick},plineb/.style={draw=black},plineg/.style={draw=lightgray}]
\node (a) at (0,0) {
\begin{tikzpicture}[scale=0.7]
\fill [fill=lightgray] (3.5,6.5) -- (4.5,6.5) -- (4.5,8.5) -- (8.5,8.5) -- (8.5,7.5) -- (7.5,7.5) -- (7.5,6.5) -- (6.5,6.5) -- (6.5,4.5) -- (5.5,4.5) -- (5.5,3.5) -- (4.5,3.5) -- (4.5,2.5) -- (3.5,2.5) -- (3.5,6.5); 
\fill [fill=lightgray] (1.5,6.5) -- (2.5,6.5) -- (2.5,8.5) -- (4.5,8.5) -- (4.5,6.5) -- (3.5,6.5) -- (3.5,2.5) -- (1.5,2.5) -- (1.5,6.5);
\fill [fill=gray] (1.5,6.5) rectangle (2.5,7.5);

\draw [plineb] (0.5,8.5) -- (8.5,8.5);
\draw [plineb] (0.5,7.5) -- (8.5,7.5);
\draw [plineb] (0.5,6.5) -- (7.5,6.5);
\draw [plineb] (0.5,5.5) -- (6.5,5.5);
\draw [plineb] (0.5,4.5) -- (6.5,4.5);
\draw [plineb] (0.5,3.5) -- (5.5,3.5);
\draw [plineb] (0.5,2.5) -- (4.5,2.5);

\draw [plineb] (0.5,8.5) -- (0.5,2.5);
\draw [plineb] (1.5,8.5) -- (1.5,2.5);
\draw [plineb] (2.5,8.5) -- (2.5,2.5);
\draw [plineb] (3.5,8.5) -- (3.5,2.5);
\draw [plineb] (4.5,8.5) -- (4.5,2.5);
\draw [plineb] (5.5,8.5) -- (5.5,3.5);
\draw [plineb] (6.5,8.5) -- (6.5,4.5);
\draw [plineb] (7.5,8.5) -- (7.5,6.5);
\draw [plineb] (8.5,8.5) -- (8.5,7.5);

\draw [post,ultra thick,white] (2,7) -- (5,7);

\node at (1.75,7.25) {$\scriptstyle x_i$};
\node at (5,7) {$\scriptstyle b$};

\node at (3.75,6.25) {$\scriptstyle x_j$};
\node at (5,6) {$\scriptstyle b+1$};
\end{tikzpicture}
};

\node (b) at (8.2,0) {
\begin{tikzpicture}[scale=0.7]
\fill [fill=lightgray] (3.5,6.5) -- (4.5,6.5) -- (4.5,8.5) -- (8.5,8.5) -- (8.5,7.5) -- (7.5,7.5) -- (7.5,6.5) -- (6.5,6.5) -- (6.5,4.5) -- (5.5,4.5) -- (5.5,3.5) -- (4.5,3.5) -- (4.5,2.5) -- (3.5,2.5) -- (3.5,6.5); 
\fill [fill=lightgray] (1.5,6.5) -- (2.5,6.5) -- (2.5,8.5) -- (4.5,8.5) -- (4.5,6.5) -- (3.5,6.5) -- (3.5,2.5) -- (1.5,2.5) -- (1.5,6.5);
\fill [fill=gray] (1.5,6.5) rectangle (2.5,7.5);

\draw [plineb] (0.5,8.5) -- (8.5,8.5);
\draw [plineb] (0.5,7.5) -- (8.5,7.5);
\draw [plineb] (0.5,6.5) -- (7.5,6.5);
\draw [plineb] (0.5,5.5) -- (6.5,5.5);
\draw [plineb] (0.5,4.5) -- (6.5,4.5);
\draw [plineb] (0.5,3.5) -- (5.5,3.5);
\draw [plineb] (0.5,2.5) -- (4.5,2.5);

\draw [plineb] (0.5,8.5) -- (0.5,2.5);
\draw [plineb] (1.5,8.5) -- (1.5,2.5);
\draw [plineb] (2.5,8.5) -- (2.5,2.5);
\draw [plineb] (3.5,8.5) -- (3.5,2.5);
\draw [plineb] (4.5,8.5) -- (4.5,2.5);
\draw [plineb] (5.5,8.5) -- (5.5,3.5);
\draw [plineb] (6.5,8.5) -- (6.5,4.5);
\draw [plineb] (7.5,8.5) -- (7.5,6.5);
\draw [plineb] (8.5,8.5) -- (8.5,7.5);

\draw [post,ultra thick,white] (2,7) -- (5,7) -- (5,6);

\node at (1.75,7.25) {$\scriptstyle x_i$};
\node at (5,6) {$\scriptstyle b+1$};

\node at (3.75,6.25) {$\scriptstyle x_j$};
\node at (5,7) {$\scriptstyle b$};
\end{tikzpicture}
};

\node at (2,-1) {$T_k$};
\node at (10.2,-1) {$T_k^*$};

\end{tikzpicture}
}

\def\FIGexchangeSix{
\begin{tikzpicture}[post/.style={->,shorten >=1pt,>=stealth',semithick},plineb/.style={draw=black},plineg/.style={draw=lightgray}]
\node (a) at (0,0) {
\begin{tikzpicture}[scale=0.7]
\fill [fill=lightgray] (3.5,6.5) -- (4.5,6.5) -- (4.5,8.5) -- (8.5,8.5) -- (8.5,7.5) -- (7.5,7.5) -- (7.5,6.5) -- (6.5,6.5) -- (6.5,4.5) -- (5.5,4.5) -- (5.5,3.5) -- (4.5,3.5) -- (4.5,2.5) -- (3.5,2.5) -- (3.5,6.5); 
\fill [fill=lightgray] (1.5,6.5) -- (2.5,6.5) -- (2.5,8.5) -- (4.5,8.5) -- (4.5,6.5) -- (3.5,6.5) -- (3.5,2.5) -- (1.5,2.5) -- (1.5,6.5);
\fill [fill=gray] (1.5,6.5) rectangle (2.5,7.5);

\draw [plineb] (0.5,8.5) -- (8.5,8.5);
\draw [plineb] (0.5,7.5) -- (8.5,7.5);
\draw [plineb] (0.5,6.5) -- (7.5,6.5);
\draw [plineb] (0.5,5.5) -- (6.5,5.5);
\draw [plineb] (0.5,4.5) -- (6.5,4.5);
\draw [plineb] (0.5,3.5) -- (5.5,3.5);
\draw [plineb] (0.5,2.5) -- (4.5,2.5);

\draw [plineb] (0.5,8.5) -- (0.5,2.5);
\draw [plineb] (1.5,8.5) -- (1.5,2.5);
\draw [plineb] (2.5,8.5) -- (2.5,2.5);
\draw [plineb] (3.5,8.5) -- (3.5,2.5);
\draw [plineb] (4.5,8.5) -- (4.5,2.5);
\draw [plineb] (5.5,8.5) -- (5.5,3.5);
\draw [plineb] (6.5,8.5) -- (6.5,4.5);
\draw [plineb] (7.5,8.5) -- (7.5,6.5);
\draw [plineb] (8.5,8.5) -- (8.5,7.5);

\draw [post,ultra thick,white] (2,7) -- (2,6) -- (3,6) -- (3,5) -- (4,5);

\node at (1.75,7.25) {$\scriptstyle x_i$};
\node at (4,5) {$\scriptstyle b$};

\node at (3.75,6.25) {$\scriptstyle x_j$};
\node at (5,6) {$\scriptstyle b+1$};
\end{tikzpicture}
};

\node (b) at (8.2,0) {
\begin{tikzpicture}[scale=0.7]
\fill [fill=lightgray] (3.5,6.5) -- (4.5,6.5) -- (4.5,8.5) -- (8.5,8.5) -- (8.5,7.5) -- (7.5,7.5) -- (7.5,6.5) -- (6.5,6.5) -- (6.5,4.5) -- (5.5,4.5) -- (5.5,3.5) -- (4.5,3.5) -- (4.5,2.5) -- (3.5,2.5) -- (3.5,6.5); 
\fill [fill=lightgray] (1.5,6.5) -- (2.5,6.5) -- (2.5,8.5) -- (4.5,8.5) -- (4.5,6.5) -- (3.5,6.5) -- (3.5,2.5) -- (1.5,2.5) -- (1.5,6.5);
\fill [fill=gray] (1.5,6.5) rectangle (2.5,7.5);

\draw [plineb] (0.5,8.5) -- (8.5,8.5);
\draw [plineb] (0.5,7.5) -- (8.5,7.5);
\draw [plineb] (0.5,6.5) -- (7.5,6.5);
\draw [plineb] (0.5,5.5) -- (6.5,5.5);
\draw [plineb] (0.5,4.5) -- (6.5,4.5);
\draw [plineb] (0.5,3.5) -- (5.5,3.5);
\draw [plineb] (0.5,2.5) -- (4.5,2.5);

\draw [plineb] (0.5,8.5) -- (0.5,2.5);
\draw [plineb] (1.5,8.5) -- (1.5,2.5);
\draw [plineb] (2.5,8.5) -- (2.5,2.5);
\draw [plineb] (3.5,8.5) -- (3.5,2.5);
\draw [plineb] (4.5,8.5) -- (4.5,2.5);
\draw [plineb] (5.5,8.5) -- (5.5,3.5);
\draw [plineb] (6.5,8.5) -- (6.5,4.5);
\draw [plineb] (7.5,8.5) -- (7.5,6.5);
\draw [plineb] (8.5,8.5) -- (8.5,7.5);

\draw [post,ultra thick,white] (2,7) -- (2,6) -- (3,6) -- (3,5) -- (4,5);

\node at (1.75,7.25) {$\scriptstyle x_i$};
\node at (4,5) {$\scriptstyle b+1$};

\node at (3.75,6.25) {$\scriptstyle x_j$};
\node at (5,6) {$\scriptstyle b$};
\end{tikzpicture}
};

\node at (0,-5) {
\begin{tikzpicture}[scale=0.7]
\fill [fill=lightgray] (3.5,6.5) -- (4.5,6.5) -- (4.5,8.5) -- (8.5,8.5) -- (8.5,7.5) -- (7.5,7.5) -- (7.5,6.5) -- (6.5,6.5) -- (6.5,4.5) -- (5.5,4.5) -- (5.5,3.5) -- (4.5,3.5) -- (4.5,2.5) -- (3.5,2.5) -- (3.5,6.5); 
\fill [fill=lightgray] (1.5,6.5) -- (2.5,6.5) -- (2.5,8.5) -- (4.5,8.5) -- (4.5,6.5) -- (3.5,6.5) -- (3.5,2.5) -- (1.5,2.5) -- (1.5,6.5);
\fill [fill=lightgray] (1.5,6.5) rectangle (2.5,7.5);
\fill [fill=gray] (1.5,7.5) rectangle (2.5,8.5);

\draw [plineb] (0.5,8.5) -- (8.5,8.5);
\draw [plineb] (0.5,7.5) -- (8.5,7.5);
\draw [plineb] (0.5,6.5) -- (7.5,6.5);
\draw [plineb] (0.5,5.5) -- (6.5,5.5);
\draw [plineb] (0.5,4.5) -- (6.5,4.5);
\draw [plineb] (0.5,3.5) -- (5.5,3.5);
\draw [plineb] (0.5,2.5) -- (4.5,2.5);

\draw [plineb] (0.5,8.5) -- (0.5,2.5);
\draw [plineb] (1.5,8.5) -- (1.5,2.5);
\draw [plineb] (2.5,8.5) -- (2.5,2.5);
\draw [plineb] (3.5,8.5) -- (3.5,2.5);
\draw [plineb] (4.5,8.5) -- (4.5,2.5);
\draw [plineb] (5.5,8.5) -- (5.5,3.5);
\draw [plineb] (6.5,8.5) -- (6.5,4.5);
\draw [plineb] (7.5,8.5) -- (7.5,6.5);
\draw [plineb] (8.5,8.5) -- (8.5,7.5);

\draw [post,ultra thick,white] (2,8) -- (4,8) -- (4,3);

\node at (4,6) {$\scriptstyle b$};

\node at (5,6) {$\scriptstyle b+1$};
\end{tikzpicture}
};

\node at (8.2,-5) {
\begin{tikzpicture}[scale=0.7]
\fill [fill=lightgray] (3.5,6.5) -- (4.5,6.5) -- (4.5,8.5) -- (8.5,8.5) -- (8.5,7.5) -- (7.5,7.5) -- (7.5,6.5) -- (6.5,6.5) -- (6.5,4.5) -- (5.5,4.5) -- (5.5,3.5) -- (4.5,3.5) -- (4.5,2.5) -- (3.5,2.5) -- (3.5,6.5); 
\fill [fill=lightgray] (1.5,6.5) -- (2.5,6.5) -- (2.5,8.5) -- (4.5,8.5) -- (4.5,6.5) -- (3.5,6.5) -- (3.5,2.5) -- (1.5,2.5) -- (1.5,6.5);
\fill [fill=lightgray] (1.5,6.5) rectangle (2.5,7.5);
\fill [fill=gray] (1.5,7.5) rectangle (2.5,8.5);

\draw [plineb] (0.5,8.5) -- (8.5,8.5);
\draw [plineb] (0.5,7.5) -- (8.5,7.5);
\draw [plineb] (0.5,6.5) -- (7.5,6.5);
\draw [plineb] (0.5,5.5) -- (6.5,5.5);
\draw [plineb] (0.5,4.5) -- (6.5,4.5);
\draw [plineb] (0.5,3.5) -- (5.5,3.5);
\draw [plineb] (0.5,2.5) -- (4.5,2.5);

\draw [plineb] (0.5,8.5) -- (0.5,2.5);
\draw [plineb] (1.5,8.5) -- (1.5,2.5);
\draw [plineb] (2.5,8.5) -- (2.5,2.5);
\draw [plineb] (3.5,8.5) -- (3.5,2.5);
\draw [plineb] (4.5,8.5) -- (4.5,2.5);
\draw [plineb] (5.5,8.5) -- (5.5,3.5);
\draw [plineb] (6.5,8.5) -- (6.5,4.5);
\draw [plineb] (7.5,8.5) -- (7.5,6.5);
\draw [plineb] (8.5,8.5) -- (8.5,7.5);

\draw [post,ultra thick,white] (2,8) -- (4,8) -- (4,6) -- (6,6);

\node at (4,6) {$\scriptstyle b$};

\node at (4,5) {$\scriptstyle b+1$};
\end{tikzpicture}
};

\node at (a.north) [above] {$T_k$};
\node at (b.north) [above] {$T_k^*$};

\end{tikzpicture}
}

\begin{document}

%Include the content
%conent.tex - 
\newbox\Adr
\setbox\Adr\vbox{
\centerline{\sc Christoph Neumann*, Robin Sulzgruber$^{\dagger}$}
\vskip18pt
\centerline{Faculty of Mathematics,
University of Vienna}
}

\title[Novelli--Pak--Stoyanovskii complexity]{A complexity theorem for the Novelli--Pak--Stoyanovskii algorithm}
\author[Christoph Neumann, Robin Sulzgruber]{\box\Adr}
\address{Faculty of Mathematics, University of Vienna, Oskar-Morgenstern-Platz~1, A-1090 Vienna, Austria}
\thanks{* Supported by the Austrian Science Foundation FWF grant Z130-N13.
\thanks{{\textdagger} Supported by the Austrian Science Foundation FWF, grant S50-N15 in the framework of the Special Research Program ``Algorithmic and Enumerative Combinatorics''.}}
\begin{abstract}
We describe two aspects of the behaviour of entries of Young tableaux during the application of the Novelli--Pak--Stoyanovskii algorithm. We derive two theorems which both imply a generalised version of a conjecture by Krattenthaler and Müller concerning the complexity of the Novelli--Pak--Stoyanovskii algorithm.
% as corollary. 
\end{abstract}
\maketitle

\secA{Introduction}\label{sec.intro}

% The Novelli--Pak--Stoyanovskii bijection \cites{PakSto1992,NPS} is known as an elegant way to prove the hook-length formula \cite{FRT} which counts the number of standard Young tableaux of a given shape $\lambda$. The bijection contains a sorting algorithm. If the tableau contains $n$ cells then this sorting algorithm transforms a permutation of $\{1,2,\dots,n\}$ into a standard Young tableau of shape $\lambda$, and each standard Young tableau is hit by the same number of permutations. Thus, this sorting algorithm can be used as a random generation algorithm of standard Young tableaux of shape $\lambda$. Motivated by this observation, Krattenthaler and Müller defined the complexity of this sorting algorithm as the average runtime of the sorting algorithm. They conjectured that the Novelli--Pak--Stoyanovskii algorithm has the same complexity no matter whether it is applied row-wise or column-wise.

An elegant proof of the hook-length formula \cite{FRT}, which counts the number of standard Young tableaux of a given shape $\lambda$, relies on the well-known Novelli--Pak--Stoyanovskii bijection \cites{PakSto1992,NPS}. 
% The Novelli--Pak--Stoyanovskii bijection \cites{PakSto1992,NPS} is known as an elegant way to prove the hook-length formula \cite{FRT} which counts the number of standard Young tableaux of a given shape $\lambda$. 
The bijection contains a sorting algorithm. If the partition $\lambda$ contains $n$ cells then this sorting algorithm transforms a permutation of $\{1,2,\dots,n\}$ into a standard Young tableau of shape $\lambda$ by means of a modified jeu de taquin, and each standard Young tableau arises from 
% is hit by
the same number of permutations. Krattenthaler and Müller defined the complexity of this sorting algorithm as its average running time. They conjectured that the Novelli--Pak--Stoyanovskii algorithm has the same complexity regardless of whether it is applied row-wise or column-wise. The complexity of this algorithm is of particular interest since the algorithm also serves as a random generator of standard Young tableaux of a given shape. For recent developments in the broader field of random properties of jeu de taquin and RSK-like algorithms the reader is referred to \cites{BFP,BGR2010,RomSni2013} and the references therein.

We consider a generalised version of the Novelli--Pak--Stoyanovskii algorithm, where the sorting order is given by an arbitrary standard Young tableau. We find that any two algorithms have the same complexity whenever they produce each standard Young tableau equally often. Since the row-wise and the column-wise Novelli--Pak--Stoyanovskii algorithms satisfy this condition, the conjecture follows. The proof relies on a recursion for exchange numbers from which one can calculate the complexity.

%In order to construct this recursion we consider how often two entries are exchanged with each other (counting all appearances while applying the sorting algorithm to all tableaux of a given shape). We find the surprising result that an entry is exchanged equally often with all greater entries -- even stronger these exchanges happen at the same positions in the Young diagram.
%
Motivated by earlier attempts to prove the conjecture of Krattenthaler and Müller we then derive a further, somewhat surprising result concerning the function that encodes the positions of the entries as they reach their maximal distance from the top left corner during the application of the sorting algorithm. This function, called the drop function, is similar to the complexity in that two algorithms have the same drop function whenever they produce each standard Young tableau equally often.
% Motivated by earlier attempts to prove the conjecture of Krattenthaler and Müller, we derive another surprising result. Namely, that the function encoding the positions of the entries, when they reach their maximal distance to the top left corner during the application of the sorting algorithm, fulfils the same property as the complexity.

In \refs{sec.alg} we review some basic definitions concerning partitions and tableaux, and explain the sorting algorithm. We present the complexity and the conjecture in \refs{sec.comp}. In \refs{sec.proof} we derive the recursion for the exchange numbers and the complexity theorem that implies the conjecture. In \refs{sec.more} we analyse the extremal positions of the entries during the application of the sorting algorithm and find a different recursion, which, whilst also implying the conjecture, is interesting in its own right. Finally, we make a few remarks and give an example in \refs{sec.rem}.

\secA{The Novelli--Pak--Stoyanovskii algorithm}\label{sec.alg}

In this section we recall some definitions and present the Novelli--Pak--Stoyanovskii algorithm. The algorithm can also be found in standard literature such as \cite{Knuth1998}*{page~70, Exercise~39} and \cite{Sagan}*{Section~3.10}.

Let $n \in \mathbb N$, and $\lambda$ be a \defstil{partition} of $n$, that is, a weakly decreasing sequence $(\lambda_1,\lambda_2,\dots)$ of nonnegative integers such that $\sum_{i=1}^{\infty}\lambda_i=n$. We identify $\lambda$ with its \defstil{Young diagram} $\{(i,j):\;i,j\in\N,\;1\leq i,\;1\leq j\leq\lambda_i\}$. We adopt the English convention of visualising a Young diagram by arranging cells like entries of a matrix such that $\lambda$ appears as a left justified array of rows with the $i$-th row containing $\lambda_i$ cells (see \reff{partition}). The \defstil{conjugate} partition $\lambda'$ of $\lambda$ corresponds to the Young diagram $\{(j, i):\;(i,j)\in\lambda\}$. Alternatively, $\lambda'=(\lambda_1',\lambda_2',\dots)$ where $\lambda_i'=\max\{j:\;\lambda_j\geq i\}$.

A \defstil{(Young) tableau} of shape $\lambda$ is a bijection $T:\lambda\to\{1,\dots,n\}$. We denote the set of all such tableaux by $\T(\lambda)$. We call $T(x)$ the \defstil{entry} of the cell $x\in\lambda$. Note that if $T\in\T(\lambda)$ is a tableau and $\sigma\in\S_n$ is a permutation, then $\sigma\circ T$ is again a tableau in $\T(\lambda)$. For convenience, we denote $T(i,j)\coloneqq T((i,j))$. Furthermore, a tableau $T\in\T(\lambda)$ is called a \defstil{standard Young tableau} if $T$ is increasing along rows from left to right, and along columns from top to bottom. That is, for each cell $(i,j)$ of $\lambda$ we have $T(i,j)<T(i+1,j)$ and $T(i,j)<T(i,j+1)$ whenever $(i+1,j)$ and $(i,j+1)$ respectively are also cells of $\lambda$. We denote the set of standard Young tableaux by $\SYT(\lambda)$.

Let $x=(i,j)$ be a cell in the Young diagram $\lambda$. We recall the usual definitions of the \defstil{arm-length} $a_{\lambda}(x)\coloneqq\lambda_i-j$, the \defstil{leg-length} $l_{\lambda}(x)\coloneqq\lambda_j'-i$, the \defstil{arm-colength} $a'_{\lambda}(x)\coloneqq j-1$, and the \defstil{leg-colength} $l'_{\lambda}(x)\coloneqq i-1$ as the numbers of cells strictly to the right, below, to the left, and above $x$ respectively. Moreover, we define the \defstil{hook length} as $h_{\lambda}(x)\coloneqq a_{\lambda}(x)+l_{\lambda}(x)+1$ and the \defstil{height} as $h'_{\lambda}(x)\coloneqq a'_{\lambda}(x)+l'_{\lambda}(x)$. We denote by $N_{\lambda}^-(x)\coloneqq\{(i-1,j),(i,j-1)\}\cap\lambda$ and $N_{\lambda}^+(x)\coloneqq\{(i+1,j),(i,j+1)\}\cap\lambda$ the sets of left and top, respectively bottom and right, neighbours of $x$ in $\lambda$. Lastly, given an entry $1\leq a\leq n$ and a tableau $T\in\T(\lambda)$ we use the notation $h'(a,T)\coloneqq h'_{\lambda}(T^{-1}(a))$.

% \myfig{partition}{We have a partition $\lambda = (5, 4, 2, 1, 1, 1)$, a cell $x = (2, 3) \in \lambda$ with $a_{\lambda}(x) = 1$, $l_{\lambda}(x) = 0$, $a'_{\lambda}(x) = 2$, $l'_{\lambda}(x) = 1$, $N_{\lambda}^-(x) = \{ (1, 3), (2, 2) \}$, $N_{\lambda}^+(x) = \{ (2, 4) \}$, and $\lambda' = (6, 3, 2, 2, 1)$.}{height=25mm}{ht}
\begin{figure}[t]
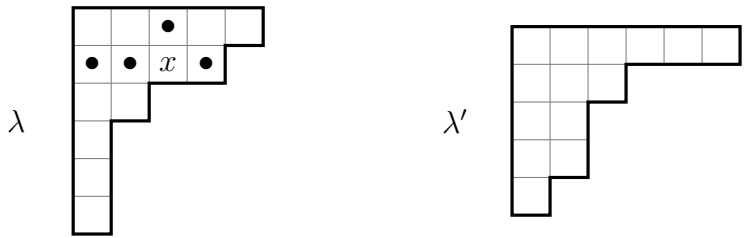

\centering
\FIGpartition
\caption{\small{We have a partition $\lambda = (5, 4, 2, 1, 1, 1)$, a cell $x = (2, 3) \in \lambda$ with $a_{\lambda}(x) = 1$, $l_{\lambda}(x) = 0$, $a'_{\lambda}(x) = 2$, $l'_{\lambda}(x) = 1$, $N_{\lambda}^-(x) = \{ (1, 3), (2, 2) \}$, $N_{\lambda}^+(x) = \{ (2, 4) \}$, and $\lambda' = (6, 3, 2, 2, 1)$.}}
\label{fig:partition}
\end{figure}

Since $T$ is a bijection we have $\abs{\T(\lambda)}=\abs{\S_n}=n!$. On the other hand, a classical result states that $f_\lambda\coloneqq\abs{\SYT(\lambda)}$ is given by the hook-length formula
\[
f_{\lambda} =  \frac {n!}{\prod_{x \in \lambda} h_{\lambda}(x)}.
\]
Novelli, Pak and Stoyanovskii prove this formula bijectively in \cite{NPS}. They define a hook function of shape $\lambda$ to be a map $H:\lambda\to\mathbb\Z$ such that $-l_{\lambda}(x)\leq H(x)\leq a_{\lambda}(x)$ for every $x\in\lambda$. Given a tableau $T\in\T(\lambda)$ Novelli, Pak and Stoyanovskii construct a standard Young tableau $W\in\SYT(\lambda)$ and a hook function $H$ on $\lambda$ such that $T$ can be recovered from the pair $(W,H)$. Since the number of hook functions of shape $\lambda$ equals $\prod_{x\in\lambda}h_{\lambda}(x)$, the hook-length formula follows.

Let $\varphi$ denote the map $T\mapsto W$ used by Novelli, Pak and Stoyanovskii. This map is given by a simple sorting algorithm which is a variation of the jeu de taquin. At each step the algorithm exchanges the entries of two adjacent cells in $\lambda$. We are interested in the average number of steps needed to transform a tableau into a standard Young tableau. We briefly recall the algorithm, which is illustrated in \reff{colalg}, in the following paragraphs. First however, we note that for all $W\in\SYT(\lambda)$ we have
\[
\abs{\varphi^{-1}(W)}=\frac{n!}{f_{\lambda}}.
\]

We impose the lexicographic order on the cells of $\lambda$ by letting $(i,j)\prec(k,l)$ if $j<l$ or $j=l$ and $i<k$. Now, set $T_0=T$.

If $T_s$ is not already a standard Young tableau, the algorithm turns to the maximal cell $x$ of $\lambda$ with respect to $\prec$ such that $N_{\lambda}^+(x)\neq\emptyset$ and $T_s(x)>\min\{T_s(y):\;y\in N_{\lambda}^+(x)\}$. We define a new tableau $T_{s+1}\coloneqq\sigma\circ T_s$, where
\[
\sigma=\big(T_s(x),\min\{T_s(y):\;y\in N_{\lambda}^+(x)\}\big)\in\S_n
\]
is a transposition. That is, we exchange the entry of $x$ with the minimal entry among its bottom and right neighbours.

It is clear that the algorithm terminates after yielding a finite sequence $(T_0, \dots, T_r)$ of $r + 1$ tableaux such that $W=T_r$ is a standard Young tableau.

% \myfig{colalg}{We have \mbox{$W = T_6 = (3,4)(1,4)(1,7)(6,5)(2,6)(2,5) \circ T$}.}{height=35mm}{ht}
\begin{figure}[t]
\centering
\FIGnps
\caption{\small{We have \mbox{$W = T_6 = (3,4)(1,4)(1,7)(5,6)(2,6)(2,5) \circ T$}.}}
\label{fig:colalg}
\end{figure}

Before we make a few observations about the nature of the Novelli--Pak--Stoyanovskii algorithm, let us consider a slightly more general setting. Clearly, for any linear order on the cells of $\lambda$ an analogous sorting algorithm can be defined. The set of linear orders on $\lambda$ can be identified with the set $\T(\lambda)$ in the following way: for each $U\in\T(\lambda)$ let
\[
x\prec_U y\deff U(x)<U(y).
\] 
By this definition $U$ is a standard Young tableau if and only if the corresponding order $\prec_U$ refines both the partial row-wise order (given by $(i,j)\prec(k,l)$ if $i=k$ and $j<l$), and the partial column-wise order (given by $(i,j)\prec(k,l)$ if $i<k$ and $j=l$). The left hand side of \reff{order} shows the standard Young tableau that gives the order used in \reff{colalg}.

From now on we will only consider orders $\prec_U$ with $U\in\SYT(\lambda)$. 

Note that there is a unique standard Young tableau $U \in \SYT(\lambda)$ such that $U(i+1,j)=U(i,j)+1$ whenever $(i,j)$ and $(i+1,j)$ are cells of $\lambda$. We call this induced order the \defstil{linear column-wise} order, which is precisely the order used by Novelli, Pak and Stoyanovskii. Analogously, there is a unique $U \in \SYT(\lambda)$ such that $U(i,j+1)=U(i,j)+1$ whenever $(i,j)$ and $(i,j+1)$ are cells of $\lambda$. We call this induced order the \defstil{linear row-wise} order.

% \myfig{order}{The linear column-wise (left) and row-wise (middle) orders on $\lambda=(3,3,1)$, and a mixed order on $\lambda=(4,3,2)$ (right).}{height=15mm}{hb}

\begin{figure}[hb]
\centering
\FIGorder
\caption{\small{The linear column-wise order (left) and the linear row-wise order (right) on $\lambda=(3,3,1)$.}}
\label{fig:order}
\end{figure}

We make the following definition.

\begin{defn}[Novelli--Pak--Stoyanovskii algorithm]\label{def:nps}
Let $n\in\N$, $\lambda$ be a partition of $n$, and $U\in\SYT(\lambda)$. For each $T\in\T(\lambda)$ let $\bphi_U(T) \coloneqq (T_0,T_1,\dots,T_r)$ where the tableaux $T_i$ arise from the algorithm with respect to $\prec_U$ as described above. We call the map $\bphi_U$ the \defstil{Novelli--Pak--Stoyanovskii algorithm corresponding to} $U$.

Accordingly, we denote $\varphi_U:\T(\lambda)\to\SYT(\lambda),\;T\mapsto T_r$. Note that $r$ is the number of steps the Novelli--Pak--Stoyanovskii algorithm needs to sort $T$ and thus depends on the tableau $T$. To make this explicit, we denote this number by $r_U(T)$.

Furthermore, we call the algorithm with respect to the linear column-wise order the \defstil{column-wise} Novelli--Pak--Stoyanovskii algorithm (see \reff{colalg}). Analogously, we define the \defstil{row-wise} Novelli--Pak--Stoyanovskii algorithm to be the algorithm with respect to the linear row-wise order.
\end{defn}

\begin{rem}\label{rem:0}
Let $U\in\SYT(\lambda)$, $T\in\T(\lambda)$, and $\bphi_U(T) = (T_0, \dots, T_{r_U(T)})$ be given by the corresponding Novelli--Pak--Stoyanovskii algorithm.

Two tableaux $T_{i-1}$ and $T_i$ differ in exactly two neighbouring cells $x$ and $y$. The transition $T_{i-1}\to T_i$ is given by the transposition $\tau_U(i,T)=(T_i(x),T_i(y))$. It follows that
\begin{align}
\label{eq:phiT}
\varphi_U(T)=\tau_U(r_U(T),T)\cdots\tau_U(1,T)\circ T.
\end{align}

Suppose $\tau_U(i,T)=(a,b)$ for some $1\leq a<b \leq n$. Then we have 
\begin{align}
\label{eq:h'}
h'(a,T_i)=h'(a,T_{i-1})-1,
\end{align}
and equivalently, $h'(b,T_i)=h'(b,T_{i-1})+1$.
\end{rem}

Next we take a closer look at the behaviour of the entries. To do so we introduce some more notation.

\begin{defn}
For every $x=(i,j)\in\lambda$ we refer to the area weakly to the right of $x$ \emph{and} weakly below $x$ as the \defstil{dropping zone} of $x$ in $\lambda$, and denote it by
\[
\out{x}=\{(k,l)\in\lambda:\;k\geq i,l\geq j\}.
\]
For any tableau $T\in\T(\lambda)$ and a subset $I \subseteq \lambda$ we say \defstil{$T$ is ordered on $I$} if $T(x)\leq\min\{T(y): y\in N_{\lambda}^+(x)\}$ for all $x \in I$ with non-empty $N_{\lambda}^+(x)$.
\end{defn}

Since $U$ is a standard Young tableau it follows that $\prec_U$ refines the partial column-wise and row-wise orders, thus we have $x\prec_U y$ for all $y\in\out{x}-\{x\}$. Hence, if $T_i(x)\neq T_{i-1}(x)$ for a cell $x$, then $T_{i-1}$ must be ordered on $\out{x}-\{x\}$.

\begin{rem} \label{rem:2}
We want to show that all exchanges of any fixed entry $b$ with an entry less than $b$ occur consecutively.

Let $x_1\prec_U x_2\prec_U\dots\prec_U x_n$ be the cells of $\lambda$, and $b=T(x_s)$ be the entry of the cell $x_s$.
Choose $0\leq i\leq r_U(T)$ minimal such that $T_{i}$ is ordered on $\{x_{s+1},\dots,x_{n}\}$, and thus in particular on $\out{x_s}-\{x_s\}$. We distinguish two cases.

Firstly, assume that $T_{i}$ is ordered on $\{x_s,\dots,x_n\}$. Then we have $T_i(y)>b$ for all $y\in\out{x_s}-\{x_s\}$. Clearly, no entry less than $b$ can be exchanged into $\out{x_s}$ thereafter. Moreover, suppose that $b$ is at some later point exchanged with the entry $T_j(x')$, then also $T_j(y)>b$ for all $y\in\out{x'}-\{x'\}$. It follows that $b$ cannot be exchanged with an entry less than $b$ for the rest of the sorting procedure.

Secondly, assume that $T_i$ is not ordered on $\{x_s\}$. Then $\tau_U(i+1,T)=(a_1,b)$ for some $1\leq a_1<b$, where $a_1$ is the entry of a bottom or right neighbour of $x_s$. Once again there are two possibilities. If $T_{i+1}$ is ordered on $\{x_s,\dots,x_n\}$ then, due to similar arguments as in the first case, $b$ cannot be exchanged with a smaller entry throughout the rest of the sorting. Otherwise, $\tau_U(i+2,T)=(a_2,b)$ for some $a_1<a_2<b$. Note that $a_1<a_2$ because $T_i$ is ordered on $\out{x_s}-\{x_s\}$.

Iterating this argument, we see that all transpositions that exchange the entry $b$ with an entry less than $b$ are processed consecutively, and in increasing order with respect to the entry less than $b$. Informally, we also say the entry $b$ drops, since each such exchange moves $b$ away from the top left corner of $\lambda$.

Now, let $b_s\coloneqq T(x_s)$ for $1\leq s\leq n$. Using the above observation we can divide $(T_1,\dots,T_r)$ into successive (possibly empty) subsequences $\bphi_U(x_{s},T)$, such that each $T_j$ belonging to $\bphi_U(x_{s},T)$ differs from $T_{j-1}$ only by a transposition of $b_{s}$ and an entry less than $b_s$. That is, each subsequence $\bphi_U(x_s, T)$ describes the dropping of the entry $b_s$.

Moreover, the length of the sequence $\bphi_U(x_s,T)$ is given by $\mu_U(s,T)-\mu_U(s+1,T)$, where $\mu_U(s,T)$ denotes the minimal integer $i$ such that $T_i$ is ordered on $\{x_s,\dots,x_n\}$. Thus, the sequence $\bphi_U(x_s,T)$ is non-empty if and only if $\mu_U(s+1,T)<\mu_U(s,T)$.
\end{rem}

Looking back to \reff{colalg} for an example, we find that $\mu_U(7,T)=0$, $\mu_U(6,T)=1$, $\mu_U(5,T)=1$, $\mu_U(4,T)=3$, $\mu_U(3,T)=3$, $\mu_U(2,T)=4$, and $\mu_U(1,T)=6=r_U(T)$. Thus, the non-empty sequences $\bphi_U(x_s,T)$ correspond to the dropping of the entries $T(x_6)=5$, $T(x_4)=6$, $T(x_2)=7$, and $T(x_1)=4$.

\secA{Complexity and the Conjecture of Krattenthaler and Müller}\label{sec.comp}

In this section we present the conjecture that motivated the current work. In order to do so we first define the complexity of a Novelli--Pak--Stoyanovskii algorithm.
%Our main objective is the number of steps performed by a Novelli--Pak--Stoyanovskii algorithm. We shall call this number the complexity of the algorithm.

\begin{defn}[Complexity]
Let $n\in\N$, $\lambda$ be a partition of $n$ and $U\in\SYT(\lambda)$. The \defstil{complexity} of the corresponding Novelli--Pak--Stoyanovskii algorithm $\bphi_U$, denoted by $C(U)$, is defined to be the average number of transitions in the sequences $\bphi_U(T)$, where $T$ ranges over $\T(\lambda)$. That is,
\[
C(U)\coloneqq\frac{1}{n!}\sum_{T\in\T(\lambda)} r_U(T).
\]
\end{defn}

%We are now able to present a conjecture by Krattenthaler and Müller which was the motivation for the present work.

\begin{conj}[\textsc{Krattenthaler, Müller}]\label{conj} Let $n\in\N$, $\lambda$ be a partition of $n$, and $U, V \in\SYT(\lambda)$ be the standard Young tableaux defining the linear column-wise and linear row-wise orders on $\lambda$ respectively. Then we have
\[
C(U)=C(V).
\]
\end{conj}

In other words, the row-wise and the column-wise Novelli--Pak--Stoyanovskii algorithms have the same complexity.

\begin{rem}\label{rem:equidist} Given a standard Young tableau $U\in\SYT(\lambda)$ we define the corresponding conjugate standard Young tableau $U'$, where $U'(j,i)\coloneqq U(i,j)$ for all $(i,j)\in\lambda$. We obtain $U'\in\SYT(\lambda')$. More precisely, this correspondence defines a bijection between $\SYT(\lambda)$ and $\SYT(\lambda')$.

Let $V\in\SYT(\lambda)$ denote the standard Young tableau defining the linear row-wise order on $\lambda$, then $V'$ induces the linear column-wise order on $\lambda'$. Thereby, also the row-wise algorithm has the property $\abs{\varphi_V^{-1}(W)} = n!/f_{\lambda}$ for all $W \in \SYT(\lambda)$.

% Additionally we want to consider orders defined by any standard Young tableau $U$ of the following form. At each step choose the top empty row or the leftmost empty column and fill it with the least possible entries (as in Figure \ref{fig:otherorder}).
Additionally we want to consider orders defined by any standard Young tableau $U$ which can be obtained by the following procedure. At each step choose the top empty row or the leftmost empty column and fill it with the least possible entries (as in Figure \ref{fig:otherorder}).
%At each step choose column or row. Depending on your choice, assign to the leftmost unlabelled cell of each row or to the topmost unlabelled cell of each column the minimal possible entry (see the right hand side of \reff{order}). 
% Each such order induces a sorting algorithm $\bphi_U$ such that $\abs{\varphi_U^{-1}(W)} = n!/f_{\lambda}$ for all $W \in \SYT(\lambda)$. 
Each such order $\prec_U$ induces a sorting algorithm with the property $\abs{\varphi_U^{-1}(W)} = n!/f_{\lambda}$ for all $W \in \SYT(\lambda)$. This can be seen from the fact that during the construction of the hook function after Novelli, Pak and Stoyanovskii, once a column of the tableau has been sorted the corresponding column of the hook function is no longer altered (see \cite{Sagan}*{Section 3.10}).

\begin{figure}[t]
\centering
\FIGotherorder
\caption{\small{A standard Young tableau defining an algorithm for which the complexity agrees with the column- or row-wise algorithm.}}
\label{fig:otherorder}
\end{figure}
\end{rem}

\secA{The proof}\label{sec.proof}

In this section we prove that a fixed entry is exchanged equally often with every greater entry and derive the recursion for the exchange numbers. The conjecture of Krattenthaler and Müller then follows.

From now on let $\bphi_U$ be the Novelli--Pak--Stoyanovskii algorithm corresponding to an arbitrary $U\in\SYT(\lambda)$. First, we observe that the transposition $(a,b)$ may occur at most once while sorting any fixed $T\in\T(\lambda)$.

We introduce a function that decides if there is an exchange of two entries at a given position. For $1\leq a,b\leq n$, $x,y\in\lambda$ and $T\in\T(\lambda)$ define
\[
m_U(a,b,x,y,T)\coloneqq\begin{cases}
1 &\quad \text{if }T_{i-1}(x)=a,T_{i-1}(y)=b,T_i(x)=b\text{, and}\\
  &\quad T_{i}(y)=a\text{ for some }1\leq i\leq r_U(T), \\
0 &\quad \text{otherwise},
\end{cases}
\]
where $\bphi_U(T)=(T_0,\dots,T_{r_U(T)})$, and 
\[
m_U(a,b,x,y)\coloneqq\sum_{T\in\T(\lambda)}m_U(a,b,x,y,T).
\]
Obviously, $m_U(a,b,x,y,T)$ and $m_U(a,b,x,y)$ both vanish unless $x$ and $y$ are neighbours. Next, we define similar functions that simply count whether $a$ and $b$ are exchanged during the algorithm without any condition on the involved cells. Let
\[
%m_U(a,b,T)\coloneqq\begin{cases} 1 &\quad\text{if there is an $1\leq i\leq r_U(T)$ such that }\tau_U(i,T)=(a,b), \\ 0 & \quad\text{else},\end{cases}
m_U(a,b,T)\coloneqq\sum_{x,y\in\lambda}m_U(a,b,x,y,T)
\]
and
\[
m_U(a,b)\coloneqq\sum_{T\in\T(\lambda)}m_U(a,b,T).
\]
Note that $m_U(a,b,T)\in\{0,1\}$ since $a$ and $b$ are exchanged at most once. Furthermore, we define the \defstil{exchange matrix} $M_U\coloneqq(m_{a,b})_{a,b}$ to be the $n\times n$-matrix with entries $m_{a,b}=m_U(a,b)$ when $1\leq a<b\leq n$ and $m_{a,b}=0$ otherwise. The essential insight of our proof is the fact that if one exchanges $a$ and $a+1$ in $T$, then up to the point when both entries $a$ and $a+1$ have dropped, the tableaux $T_i$ arising in the Novelli--Pak--Stoyanovskii algorithm differ at most by the transposition $(a,a+1)$. We use this fact to prove the following central proposition.

\begin{prop}\label{lemma} Let $n\in\N$, $\lambda$ be a partition of $n$ and $U\in\SYT(\lambda)$ a standard Young tableau. For all $a,b,c\in\N$ with $1\leq a<b\leq n$ and $1\leq a<c\leq n$, and all $x,y\in\lambda$ we have
\begin{align}\label{eq:mabxy}
m_U(a,b,x,y)=m_U(a,c,x,y).
\end{align}
Furthermore, we have the symmetry
\begin{align}\label{eq:msym}
m_U(a,b,x,y)=m_U(b,a,y,x).
\end{align}
Finally, we have
\begin{align}\label{eq:mab}
m_U(a,b)=m_U(a,c).
\end{align}
\end{prop}

Hence, for $1\leq a<b\leq n$ we denote the \defstil{exchange numbers} by
\[
m_U(a)=m_U(a,b),
\]
and the \defstil{local exchange numbers} by
\[
m_U(a,x,y)=m_U(a,b,x,y).
\]

\begin{proof}
The symmetry in \eqref{eq:msym} is evident. Moreover, \eqref{eq:mab} follows from \eqref{eq:mabxy} by summation over all pairs of cells $x,y\in\lambda$. To show \eqref{eq:mabxy} it suffices to consider the case $c=b+1$. Let $T\in\T(\lambda)$ and $x_1\prec_U x_2\prec_U\cdots\prec_U x_n$ be the cells of $\lambda$. Now, choose $1\leq i,j\leq n$ such that $T(x_i)=b$ and $T(x_j)=b+1$. Set $\sigma=(b,b+1)$ and $T^*=\sigma\circ T$. Without loss of generality we may assume that $i<j$. For convenience we denote by $T^*_k$ the tableaux that appear during the application of the Novelli--Pak--Stoyanovskii algorithm to $T^*$, and the corresponding transposition is denoted $\tau^*_k$.

Obviously, for $0\leq k\leq\mu_U(j+1,T)$ we have $T_k=\sigma\circ T^*_k$, since none of the involved transitions are influenced by the entries of $x_i$ or $x_j$. See \reff{exchange1}.

\begin{figure}[ht]
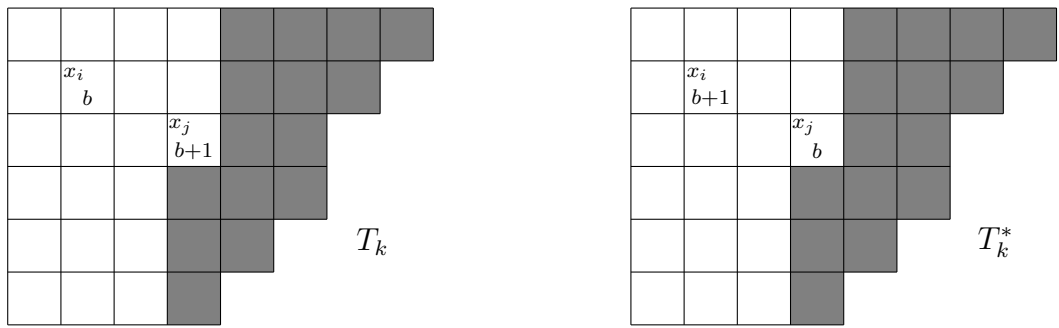

\centering
\FIGexchangeOne
\caption{\small{While $0\leq k\leq\mu_U(j+1,T)$, the entries $b,b+1$ are not moved.}}
\label{fig:exchange1}
\end{figure}

Now, we consider the dropping of the entry of $x_j$. Since all entries different from $b$ and $b+1$ are either less than both $b$ and $b+1$ or greater than both $b$ and $b+1$, the dropping path of the entry of $x_j$ does not depend on whether it is $b$ or $b+1$. Hence, also for $\mu_U(j+1,T)<k\leq\mu_U(j,T)$ we have $T_k=\sigma\circ T^*_k$. See \reff{exchange2}.

\begin{figure}[hb]
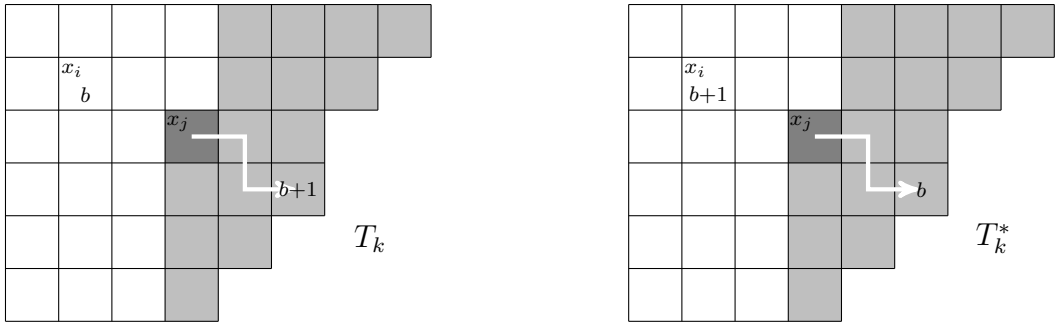

\centering
\FIGexchangeTwo
\caption{\small{While $\mu_U(j+1,T)<k\leq\mu_U(j,T)$, the first entry among $b$ and $b+1$ drops.}}
\label{fig:exchange2}
\end{figure}

By the same argument the dropping paths in $T$ and $T^*$ are the same for the initial entries of $x_l$, for all $x_i\prec_U x_l\prec_U x_j$. Hence also for $\mu_U(j,T)<k\leq\mu_U(i+1,T)$ we have $T_k=\sigma\circ T^*_k$. See \reff{exchange3}.

\begin{figure}[ht]
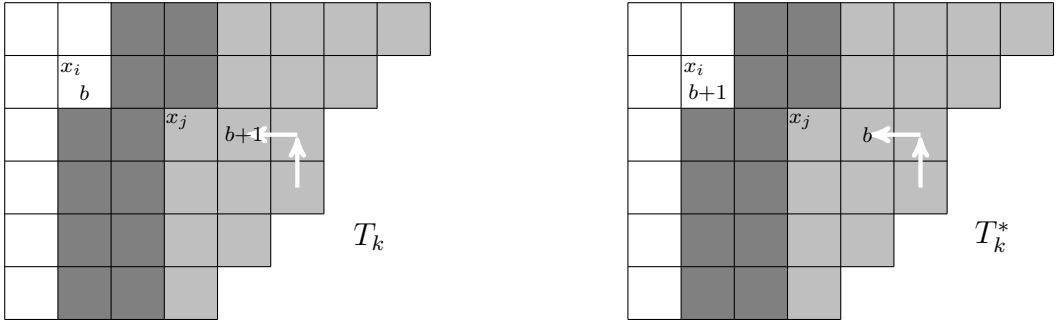

\centering
\FIGexchangeThree
\caption{\small{While $\mu_U(j,T)<k\leq\mu_U(i+1,T)$, it is possible that the entry among $b$ and $b+1$ which has already dropped can be move upwards again.}}
\label{fig:exchange3}
\end{figure}

Lastly, we consider the dropping of the initial entry of $x_i$. Since $T_{\mu_U(i+1,T)}=\sigma\circ T^*_{\mu_U(i+1,T)}$ the dropping paths will again agree, unless $b$ and $b+1$ are exchanged at some point (that is, if $\sigma$ occurs as transition). This situation may arise only if $\tau^*_{\mu_U(i,T)+1}=\tau^*_{\mu_U(i,T^*)}=\sigma$ (that is, the very last transition of the dropping of $b+1$ in $T^*$ may be $\sigma$). Hence, for $\mu_U(i+1,T)<k\leq\mu_U(i,T)$ we have once more $T_k=\sigma\circ T^*_k$. See Figures \ref{fig:exchange4} and \ref{fig:exchange5}.

\begin{figure}[hb]
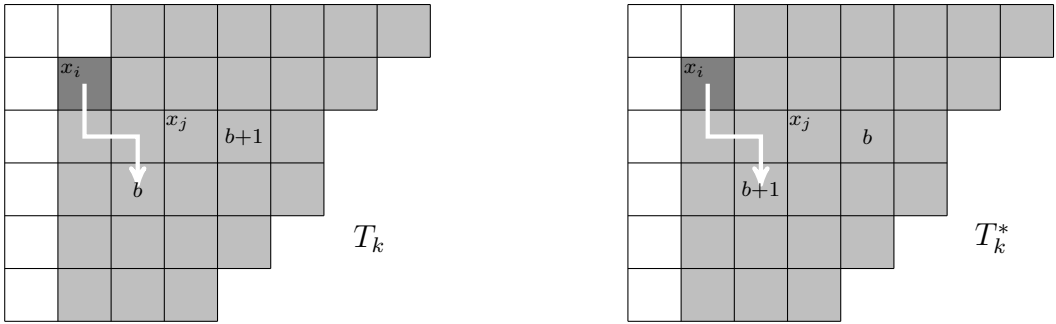

\centering
\FIGexchangeFour
\caption{\small{While $\mu_U(i+1,T)<k\leq\mu_U(i,T)$, the second entry among $b$ and $b+1$ drops.}}
\label{fig:exchange4}
\end{figure}

\begin{figure}[ht]
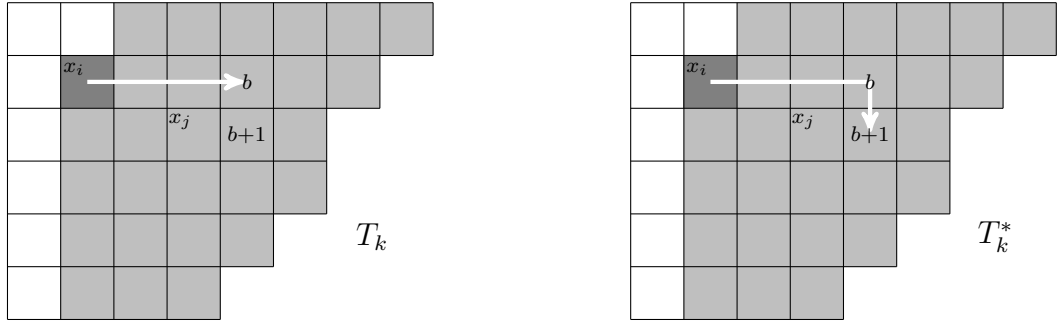

\centering
\FIGexchangeFive
\caption{\small{The only stage where the dropping paths may differ is the last exchange.}}
\label{fig:exchange5}
\end{figure}

To summarise, $T_k=\sigma\circ T^*_k$ for all $0\leq k\leq\mu_U(i,T)$. The rest of the sequences $\bphi_U(T)$ and $\bphi_U(T^*)$ may differ heavily (see \reff{exchange6}). However, we know that all transitions $\tau_k$ and $\tau^*_k$ that exchange $b$ or $b+1$ with an entry $a<b$ happen solely up to the index $\mu_U(i,T)$. Hence, the dropping path of $b$ in $T$ agrees exactly with the dropping path of $b+1$ in $T^*$ and the dropping path of $b+1$ in $T$ agrees exactly with the dropping path of $b$ in $T^*$. Therefore, we have
\[
m_U(a,b,x,y,T)=m_U(a,b+1,x,y,T^*)
\]
and
\[
m_U(a,b+1,x,y,T)=m_U(a,b,x,y,T^*).
\]
Since $T\mapsto\sigma\circ T$ is an involution, summation over $T\in\T(\lambda)$ yields \eqref{eq:mabxy}, and the proof is complete.
\end{proof}

\begin{figure}[ht]
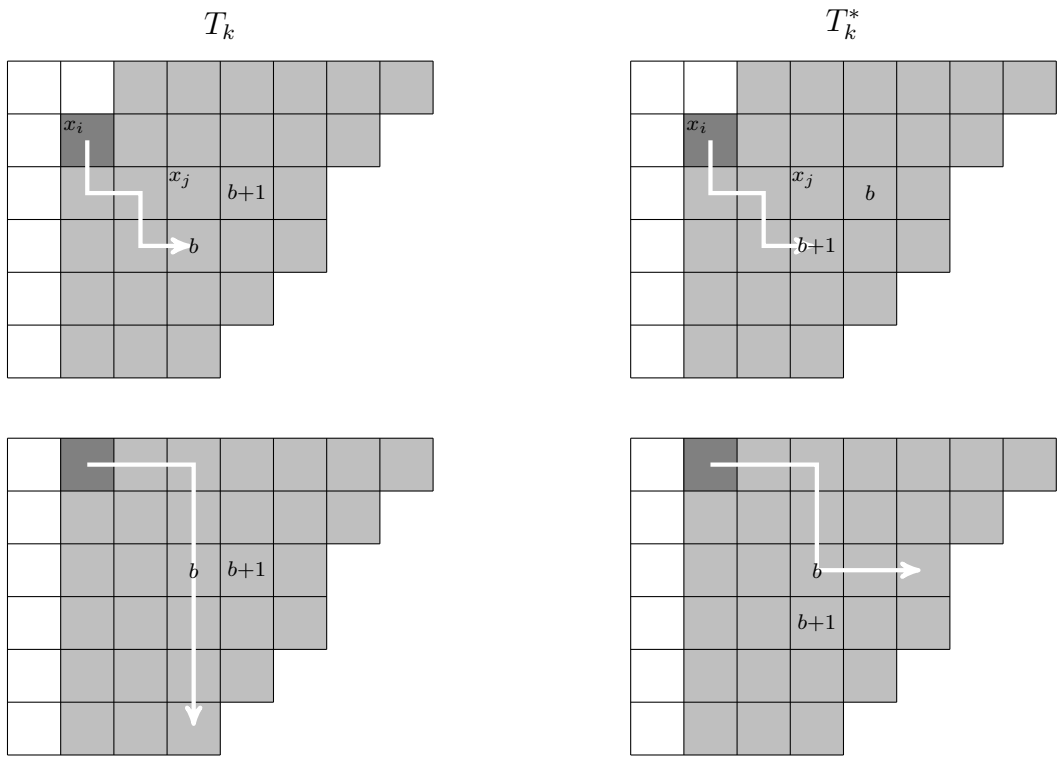

\centering
\FIGexchangeSix
\caption{\small{Suppose the entries $b$ and $b+1$ end up in diagonally adjacent positions after they have both dropped. Then exchanging $b$ and $b+1$ can divert a large entry dropping at a later point.}}
\label{fig:exchange6}
\end{figure}

Note that we can express the complexity in terms of exchange numbers in the following way
\[
C(U)=\frac{1}{n!}\sum_{a=1}^n (n-a)m_U(a).
\]

To state and prove the recursion for the exchange numbers we need some more notation.

\begin{defn}
Let $n\in\N$, $\lambda$ be a partition of $n$ and $U\in\SYT(\lambda)$. For $W\in\SYT(\lambda)$ we define the \defstil{multiplicity} of $W$ with respect to $U$ as
\[
z_U(W)\coloneqq\abs{\big\{T\in\T(\lambda):\;\varphi_U(T)=W\big\}}
\]
and the \defstil{distribution vector} of $U$ as
\[
Z_U\coloneqq\big(z_U(W)\big)_{W\in\SYT(\lambda)}.
\]
Moreover, we call $\bphi_U$ \defstil{uniformly distributed} if all entries of $Z_U$ agree. That is, for all $W\in\SYT(\lambda)$ we have
\[
z_U(W)=\frac{n!}{f_\lambda}.
\]
Before the application of the Novelli--Pak--Stoyanovskii algorithm every entry has a distance from the top left corner. Summing up these distances over all tableaux we define the \defstil{total initial height} of the entry $b$ as
\[
\alpha_\lambda(b)\coloneqq\sum_{T\in\T(\lambda)}h'(b,T).
\]
After the application the entry has taken its terminal position in a standard Young tableau with a (different) distance from the top left corner. Summing up these distances over all initial tableaux we define the \defstil{total terminal height} of the entry $b$ as
\[
\omega_U(b)\coloneqq\sum_{T\in\T(\lambda)}h'(b,\varphi_U(T)).
\]
\end{defn}

\begin{rem}\label{rem:omega}
The above parameter $\alpha_\lambda(b)$ does not depend on $b$, that is,
\begin{align*}
\alpha_\lambda(b)&=\sum_{T\in\T(\lambda)}h'(b,T)\\
&=\sum_{T\in\T(\lambda)}\bigg(\sum_{x\in\lambda,\,T(x)=b}h'_\lambda(x)\bigg)\\
&=\sum_{x\in\lambda}\bigg(\sum_{T\in\T(\lambda),\,T(x)=b}h'_\lambda(x)\bigg)\\
&=(n-1)!\sum_{x\in\lambda}h'_\lambda(x).
\end{align*}
Hence, we denote it simply by $\alpha_\lambda$. Note that we could calculate $\alpha_\lambda$ also as a sum of hook lengths $\alpha_\lambda=(n-1)!\left(-n+\sum_{x\in\lambda}h_\lambda(x)\right)$ or even in terms of $\lambda_i$ as $\alpha_\lambda=(n-1)!\sum_{i\in\N}\left(\tbinom{\lambda_i}{2}+(i-1)\lambda_i\right)$.

Furthermore, $\omega_U(b)$ does not depend on $\varphi_U$ but rather on $Z_U$, that is,
\begin{align*}
\omega_U(b)&=\sum_{T\in\T(\lambda)}h'(b,\varphi_U(T))\\
&=\sum_{W\in\SYT(\lambda)}\bigg(\sum_{T\in\T(\lambda),\,\varphi_U(T)=W}h'(b,W)\bigg)\\
&=\sum_{W\in\SYT(\lambda)}z_U(W)h'(b,W).
\end{align*}
\end{rem}

We are now in good shape to derive the aforementioned recursion.

\begin{thm}[Exchange numbers]\label{main}
Let $n\in\N$, $\lambda$ be a partition of $n$ and $U\in\SYT(\lambda)$. Then for $1\leq b\leq n$ we have the recursion
%Let $n\in\N$, $\lambda$ be a partition of $n$, $U\in\SYT(\lambda)$ and $\bphi_U$ the Novelli--Pak--Stoyanovskii algorithm corresponding to $U$. Then for $1\leq b\leq n$ we have the recursion
\begin{align}\label{eq:rec}
(n-b)\,m_U(b)=\alpha_\lambda-\omega_U(b)+\sum_{a=1}^{b-1}m_U(a).
\end{align}
\end{thm}
 
\begin{proof}
From \eqref{eq:h'} we conclude
\[
\alpha_\lambda+\sum_{a=1}^{b-1}m_U(a,b)-\sum_{c=b+1}^n m_U(b,c)=\omega_U(b),
\]
which says that the starting height plus the steps away from the top left corner minus the steps towards the top left corner equals the total terminal height. Using \eqref{eq:mab} we obtain
\[
\alpha_\lambda+\sum_{a=1}^{b-1}m_U(a)-\sum_{c=b+1}^n m_U(b)=\omega_U(b),
\]
and hence
\[
\alpha_\lambda+\sum_{a=1}^{b-1}m_U(a)-(n-b)m_U(b)=\omega_U(b).
\]
This concludes the proof.
\end{proof}

Note that for $b=1$ the sum on the right hand side of \eqref{eq:rec} is empty. Moreover, $\omega_U(1)=0$ since the entry $1$ will always end up in the top left corner of the standard Young tableau. The recursion, therefore, yields its own initial condition,
\[
m_U(1)=\frac{\alpha_\lambda}{n-1}.
\]
% Therefore, we can recursively compute the exchange matrix. We observe that the exchange matrix depends only on $\lambda$ and $Z_U$ rather than on $U$. We obtain the following corollaries.
We now deduce the main result of this paper.

\begin{cor}[Complexity Theorem]\label{complexitythm}
Let $n\in\N$, $\lambda$ be a partition of $n$ and $U,V\in\SYT(\lambda)$ such that $Z_U=Z_V$. Then we have
\[
C(U)=C(V).
\]
In particular, the row-wise and column-wise Novelli--Pak--Stoyanovskii algorithms have the same complexity, that is, the conjecture of Krattenthaler and Müller holds.
%In the situation of \reft{main} we can calculate the complexity of $\bphi_U$ as
%\[
%C(U)=\frac{1}{n!}\sum_{a=1}^n(n-a)m_U(a).
%\]
%Moreover, let $V\in\SYT$ be such that c. Then we have
%\[
%C(U)=C(V).
%\]
\end{cor}

\begin{proof} We observed that we can compute the complexity $C(U)$ from the exchange numbers $m_U(a)$, which may in turn be computed from the terminal heights $\omega_U(a)$ using the above recursion. By Remark \ref{rem:omega} the numbers $\omega_U(a)$ only depend on $\lambda$ and $Z_U$ rather than on $U$.

The claim follows since the row-wise and column-wise algorithms are both uniformly distributed (see Remark \ref{rem:equidist}).
\end{proof}

\begin{cor}
Let $n\in\N$, $\lambda$ be a partition of $n$ and $U\in\SYT(\lambda)$ such that $\bphi_U$ is uniformly distributed. Then we obtain the same recursion as in Theorem \ref{main} with the specialisation
\[
\omega_U(b)=\frac{n!}{f_\lambda}\sum_{W\in\SYT(\lambda)}h'(b,W).
\]
\end{cor}

% Since the row-wise and column-wise Novelli--Pak--Stoyanovskii algorithm are both uniformly distributed, we have in particular that the conjecture of Krattenthaler and Müller holds.

\secA{Intermediate targets of entries}\label{sec.more}

In this section we define the drop function that counts those tableaux in which a certain cell is the one farthest from the top left corner among all cells that contain a specific entry at some point during the application of the Novelli--Pak--Stoyanovskii algorithm. We also derive another recursion that implies this drop function depends only on $Z_U$ and not $U$.

As mentioned earlier, during the application of a Novelli--Pak--Stoyanovskii algorithm each entry first raises its height to a maximum and then lowers it to its final height. Therefore, for each entry $1 \leq b \leq n$ and tableau $T\in\T(\lambda)$, there is a unique cell of $\lambda$ with maximal height which contains $b$ at some point during the sorting of $T$. We denote this maximal height by $\beta_U(b, T)$.

Let $\bphi_U(T) = (T_0, \dots, T_{r_U(T)})$, and suppose $b = T(x_s)$ for some $1 \leq s \leq n$, where $x_1 \prec_U \cdots \prec_U x_n$ are the cells of $\lambda$ ordered with respect to $U$, then the maximal height is given by $\beta_U(b, T) = h'(b, T_{\mu_U(s, T)})$.

Summation over $T$ yields the statistic
\[
\beta_U(b)\coloneqq\sum_{T\in\T(\lambda)}\beta_U(b , T).
\]
Using the notation from \reft{main} we can immediately derive the relation
\[
C(U)= \frac{1}{n!} \sum_{b=1}^n\big(\beta_U(b) - \alpha_\lambda\big).
\]
To see this note that $\beta_U(b)-\alpha_\lambda$ counts the number of exchanges of the entry $b$ with a smaller entry. Each exchange is therefore counted exactly once.

Let $U,V\in\SYT(\lambda)$ be the standard Young tableaux defining the row-wise and the column-wise Novelli--Pak--Stoyanovskii algorithms. \refc{conj} would follow from $\beta_U(b)=\beta_V(b)$ for all entries $1 \leq b \leq n$.
% ??? Note that this follows from the fact, that the arising exchange matrices agree. Nevertheless, 
This approach raises a natural further question, namely, which cells will a given entry actually drop to? The rest of this section is devoted to answering this question.

\begin{defn}[Drop function]
Let $n\in\N$, $\lambda$ be a partition of $n$ and $U\in\SYT(\lambda)$ define a Novelli--Pak--Stoyanovskii algorithm $\bphi_U$. Given an entry $1 \leq b \leq n$, a cell $x \in \lambda$ and a tableau $T\in\T(\lambda)$ we define 
\[
d_U(b, x, T) \coloneqq \begin{cases} 1 \quad & \text{if }h'_{\lambda}(x) = \beta_U(b, T) \text{ and } T_i(x) = b \\ & \text{ for some } 1 \leq i \leq r_U(T), \\ 0 \quad & \text{otherwise,} \end{cases}
\]
where $\bphi_U(T) = (T_0, \dots, T_{r_U(T)})$.
The \defstil{drop function} $d_U(b,x)$ counts how often the entry $b$ drops to the cell $x$, when all tableaux in $T(\lambda)$ are considered, that is,
\[
d_U(b,x) \coloneqq \sum_{T \in \T(\lambda)} d_U(b, x, T).
\]
\end{defn}

Note that if $b$ is never exchanged with an entry less than $b$, then it drops to its starting position. In particular, $d_U(1,x)=(n-1)!$ for all $x\in\lambda$. Moreover, $\sum_{x\in\lambda}d_U(b,x)=n!$ for all $b$.
% Note that if $b$ is never exchanged with a label less than $b$, then it drops to its starting position. Thus, $\sum_{x\in\lambda}d_U(b,x)=n!$ for all $b$, and in particular $d_U(1,x)=(n-1)!$ for all $x\in\lambda$.

In order to calculate the drop function we need to define intermediate quantities which are suitable to construct a recursion.

\begin{defn}[Signed exit number]
Let $n\in\N$, $\lambda$ be a partition of $n$ and $U\in\SYT(\lambda)$. The \defstil{signed exit number} of the entry $b$ at the cell $x$ is defined as
% With the notation of \reft{main} the \defstil{signed exit number} of the entry $b$ at the cell $x$ is defined as
\[
\Delta_U(b,x)\coloneqq\sum_{y\in N^-(x)}m_U(b,x,y)-\sum_{y\in N^+(x)} m_U(b,y,x).
\]
Furthermore, let
\[
\omega_U(b, x, T)\coloneqq\begin{cases}1&\quad\text{if }\varphi_U(T)(x)=b,\\
0&\quad\text{otherwise}.
\end{cases}
\]
Then
\[
\omega_U(b, x) \coloneqq \sum_{T \in \T(\lambda)} \omega_U(b, x, T)
\]
counts the number of tableaux in which the terminal position of $b$ after the application of the sorting algorithm is $x$.
\end{defn}

\begin{thm}[Signed exit numbers]\label{thm:slope}
Let $n \in \mathbb N$, $\lambda$ be a partition of $n$ and $U \in \SYT(\lambda)$. For all cells $x\in\lambda$ and for all entries $1\leq b\leq n$ we have the recursion
\[
(n-b)\, \Delta_U(b,x)=(n-1)!-\omega_U(b,x)+\sum_{a=1}^{b-1}\Delta_U(a,x).
\]
\end{thm}

\begin{proof}
Let $N(x)=N_{\lambda}^-(x)\cup N_{\lambda}^+(x)$ be the set of adjacent cells of $x$ in $\lambda$. Fix an entry $b$ and a cell $x$. The number of tableaux $T$ such that $\varphi_U(T)(x)=b$ is obtained by adding the number of tableaux in which $b$ starts in $x$ and the number of times $b$ is exchanged to $x$, and subtracting the number of times $b$ is exchanged away from $x$. That is,
\[
\omega_U(b,x) = (n-1)!+\sum_{a\neq b}\sum_{y\in N(x)}\Big(m_U(b,a,y,x)-m_U(b,a,x,y)\Big).
\]
Using \refp{lemma}, the double sum in the above equation becomes

\begin{align*}
&\sum_{a=1}^{b-1}\bigg(\sum_{y\in N^-_\lambda(x)}m_U(a,b,x,y)-\sum_{y\in N_\lambda^+(x)}m_U(a,b,y,x)\bigg)+\\
&\quad+\sum_{c=b+1}^n\bigg(\sum_{y\in N_\lambda^+(x)}m_U(b,c,y,x)-\sum_{y\in N_\lambda^-(x)}m_U(b,c,x,y)\bigg)\\
=&\sum_{a=1}^{b-1}\bigg(\sum_{y\in N_\lambda^-(x)}m_U(a,x,y)-\sum_{y\in N_\lambda^+(x)}m_U(a,y,x)\bigg)+\\
&\quad+\sum_{c=b+1}^n\bigg(\sum_{y\in N_\lambda^+(x)}m_U(b,y,x)-\sum_{y\in N_\lambda^-(x)}m_U(b,x,y)\bigg)\\
=&\sum_{a=1}^{b-1}\Delta_U(a,x)-(n-b)\Delta_U(b,x).
\end{align*}
This completes the proof.
\end{proof}

As before, this recursion generates its own initial condition. Hence we can recursively compute $\Delta_U(a,x)$ and use it to determine the drop function.

\begin{cor}[Drop Theorem]
With the notation of \reft{thm:slope}, the drop function can be derived from the signed exit numbers as
\[
d_U(b,x)=(n-1)!+\sum_{a=1}^{b-1}\Delta_U(a,x).
\]

Furthermore, if $U,V\in\SYT(\lambda)$ define equidistributed Novelli--Pak--Stoyanovskii algorithms, that is, $Z_U=Z_V$, then for all $x\in\lambda$ and $1\leq b\leq n$ we have $d_U(b,x)=d_V(b,x)$.
\end{cor}

\begin{proof}
In order for the entry $b$ to drop to a cell $x$ it must either enter $x$ from $N^-_\lambda(x)$ or start in $x$, and additionally $b$ must never leave $x$ towards $N^+_\lambda(x)$. Summing over $T\in\T(\lambda)$ we find that $b$ starts at $x$ exactly $(n-1)!$ times, it enters $x$
\[
\sum_{T\in\T(\lambda)}\sum_{a=1}^{b-1}\sum_{y\in N_\lambda^-(x)}m_U(b,a,y,x,T)
\]
times from the left or from above, and leaves it
\[
\sum_{T\in\T(\lambda)}\sum_{a=1}^{b-1}\sum_{y\in N_\lambda^+(x)}m_U(b,a,x,y,T)
\]
times to the right or below. Hence we obtain
\begin{align*}
d_U(b,x)
% &=(n-1)!+\sum_{T\in\T(\lambda)}\sum_{a=1}^{b-1}\sum_{y\in N_\lambda^-(x)}m_U(b,a,y,x;T)\\
% &\phantom{=(n-1)!}-\sum_{T\in\T(\lambda)}\sum_{a=1}^{b-1}\sum_{y\in N_\lambda^+(x)}m_U(b,a,x,y;T)\\
&=(n-1)!+\sum_{a=1}^{b-1}\bigg(\sum_{y\in N_\lambda^-(x)}m_U(a,b,x,y)-\sum_{y\in N_\lambda^+(x)}m_U(a,b,y,x)\bigg) \\
&=(n-1)!+\sum_{a=1}^{b-1}\bigg(\sum_{y\in N_\lambda^-(x)}m_U(a,x,y)-\sum_{y\in N_\lambda^+(x)}m_U(a,y,x)\bigg) \\
&=(n-1)!+\sum_{a=1}^{b-1}\Delta_U(a,x).
\end{align*}
The second claim follows from the fact that $\omega_U(b, x)$ depends only on $Z_U$ rather than on $U$. Thus, by \reft{thm:slope}, the signed exit numbers and the drop function depend only on $Z_U$.
\end{proof}

\secA{Remarks}\label{sec.rem}

The above two sections both come to the conclusion that the objects of study (the complexity, the signed exit number and the drop function) depend on $Z_U$ rather than on $U$. The distribution vector appeared implicitly in earlier work by Fischer \cite{F}. We state her result below in the form of a remark.

\begin{rem}[Fischer 2002]
Let $n\in\N$, $\lambda$ be a partition of $n$ and $(z_U(W))_{U,W\in\SYT(\lambda)}$ be the matrix of multiplicities of $W$ with respect to $U$. Then $(z_U(W))_{U,W\in\SYT(\lambda)}$ is symmetric.
\end{rem}

Moreover, we would like to consider further possible generalisations.

\begin{rem}
As in \cite{F}, our arguments may be generalised to the skew and shifted case. But since the row-wise and column-wise orders may yield different distribution vectors, the conjecture of Krattenthaler and Müller does not apply. 
\end{rem}

We also observe the following property of the signed exit number.

\begin{rem}
The signed exit number $\Delta_U(b,x)$ measures the difference between how often $b$ leaves $x$ towards $N^-_\lambda(x)$ and how often it enters $x$ from $N^+_\lambda(x)$. Thus it measures how strong $x$ is as a source for pushing $b$ towards the top left corner, respectively how strongly $x$ acts as a sink if $\Delta_U(x,b)$ is negative. For our purposes (namely computing the drop function) it would be enough to consider it as a formal quantity. Nevertheless, we want to remark that summing over all $x\in\lambda$ it counts every exchange of $b$ once with positive and once with negative sign. Hence, we have for all $1\leq b\leq n$
\[
\sum_{x\in\lambda}\Delta_U(b,x)=0.
\]  
\end{rem}

In a special case we can actually compute the values of the drop function.

\begin{ex}[The drop function for a single lined Young diagram]
Let $n\in\N$ and consider the partition $\lambda=(n)$. We can treat a cell $x$ as a single index $1\leq x\leq n$. Moreover, there is only one standard Young tableau $U\in\SYT(\lambda)$. We shall therefore index all functions by $n$ instead of $U$. To calculate the drop function explicitly without using the signed exit numbers 
% from the previous section -- we chose this way of calculation because it gives room for another remark. 
we introduce the \defstil{partial drop function}, which counts the number of tableaux in which the entry $a$ drops from the starting position $x$ to $y$,
\[
d_{n}(a,x,y)=\abs{\big\{T\in\T(\lambda):\;T(x)=a, d_n(a,y,T)=1\big\}}.
%=T_{\nu_n(x,T)}(y)\}}.
\]
A tableau with $T(x)=a$ fulfils $d_n(a,y,T)=1$ if and only if exactly $n-y$ of the $n-x$ cells to the right of $x$ are occupied by entries greater than $a$. Of the $x-1$ cells to the left of $x$, exactly $n-a-(n-y)$ contain an entry greater than $a$. Thus the partial drop functions are given explicitly by
% First, note that we may give the partial drop functions explicitly as
\[
d_n(a,x,y)=\binom{x-1}{y-a}\binom{n-x}{n-y}(a-1)!\,(n-a)!\,.
\]
% Thus, the desired result can be obtained using the Chu--Vandermonde identity. Instead we shall derive a combinatorial recursion
% Thus, an explicit formula for $d_n(a,x)$ could be obtained using the Chu--Vandermonde identity \cite{GKP1990}*{p.169~(5.22)}. Instead 
We derive the recursion
\[
d_{n}(a+1,x,y)=\frac{y-a}{n-a}\, d_{n}(a,x,y)+\frac{n-y+1}{n-a}\, d_{n}(a,x,y-1)
\]
as follows. Exchanging $a$ and $a+1$ defines a bijection between the tableaux in which $a+1$ drops from $x$ to $y$ and the tableaux in which either $a$ drops from $x$ to $y$ and $a+1$ is to the left of $a$, or $a$ drops from $x$ to $y-1$ and $a+1$ starts to its right.

The left summand corresponds to the first case since here $y-a$ of the $n-a$ entries larger than $a$ must start to its left. Analogously, the right summand corresponds to the latter case since $n-y+1$ entries larger than $a$ start to its right.

Summation over $x$ yields the recursion
\[
d_{n}(a+1, y)=\frac{y-a}{n-a}\, d_{n}(a, y)+\frac{n-y+1}{n-a}\, d_{n}(a, y-1).
\]
Since the entry $1$ always drops to its initial position, we have the initial condition
\[
d_{n}(1, x)=(n-1)!
\]
for all $x\in\lambda$. A straightforward calculation yields
\[
d_{n}(a, x)=\begin{cases}\frac{n!}{n-a+1}&\quad\text{if }x\geq a, \\0&\quad\text{else.}\end{cases}
\]
\end{ex}

\begin{rem}
The above example perhaps raises the hope of finding a general recursion directly for the drop function that uses partial drop functions rather than signed exit numbers. 
Unfortunately, such an approach seems unlikely to work except in very specific examples. Indeed, for general $\lambda$ the partial drop functions really do depend on $U$ and not only on $Z_U$ whereas the signed exit numbers are much better behaved.
% For at least two reasons this seams unlikely.
% \begin{itemize}
% \item T
% \item The above calculation is based on the fact that $a+1$ drops to $y+1$ if the exchange with $a$ moves $a + 1$ to a smaller cell (with respect to some standard Young tableau). In the general case it could not just drop to $N^+_\lambda(y)$ but also to some other positions which generates a case that is more difficult to treat.
% \end{itemize}
\end{rem}

Nevertheless, we observe that for $\lambda=(n)$ we have
\[
\gcd\{d_{n}(a, x):\;1\leq a,x\leq n\}
=\gcd\left\{\frac{n!}{n-a+1}:\;1\leq a\leq n\right\}
=\frac{n!}{\lcm\{1,\dots,n\}}.
\]
Surprisingly, in a computer experiment for the row-wise Novelli--Pak--Stoyanovskii algorithm applied to other tableaux of small shape this equality held as well. For several other tableaux the greatest common divisor was still a factor of the right hand side. Hence, we close with the following conjecture.

\begin{conj}
Let $n\in\N$, $\lambda$ be a partition of $n$ and $U\in\SYT(\lambda)$ such that the corresponding Novelli--Pak--Stoyanovskii algorithm is uniformly distributed. Then 
\[
\frac{n!}{\lcm\{1,\dots,n\} \cdot \gcd\{d_U(a,x):\;1\leq a\leq n,x\in\lambda\}}\in\N.
\]
\end{conj}

\section*{Acknowledgement}
We are thankful to Theresia Eisenkölbl, Ilse Fischer, Christian Krattenthaler and Henri Mühle for their helpful comments. We also wish to thank the anonymous referees for their thorough reading of this paper and many useful suggestions.

The research was funded by the FWF grants Z130-N13 and S50-N15.

%Include the references
\bibliographystyle{plain}
\bibliography{complexity}

\end{document}